\documentclass{amsart}

\usepackage{graphicx}
\usepackage[centertags]{amsmath}
\usepackage{amsfonts}
\usepackage{amssymb}
\usepackage{amsthm}
\usepackage{newlfont}
\usepackage{hyperref}

\newtheorem{thm}{Theorem}[section]
\newtheorem{cor}[thm]{Corollary}
\newtheorem{lem}[thm]{Lemma}
\newtheorem{prop}[thm]{Proposition}

\theoremstyle{definition}
\newtheorem{defn}[thm]{Definition}

\theoremstyle{remark}

\numberwithin{equation}{section}

\newcommand{\N}{\mathbb{N}}
\newcommand{\Z}{\mathbb{Z}}
\newcommand{\R}{\mathbb{R}}

\newcommand{\B}[2]{\mathbf{B}_{#1}}
\newcommand{\Sp}[2]{\mathbf{S}_{#1}}
\newcommand{\C}{\mathcal{C}}
\newcommand{\Pre}{\mathbb{P}}

\newcommand{\E}{\mathfrak{E}}
\newcommand{\M}{\mathfrak{M}}

\newcommand{\acts}{\curvearrowright}
\newcommand{\symd}{\triangle}

\newcommand{\sH}{\mathrm{H}}

\newcommand{\h}{h}

\renewcommand{\:}{\,:\,}

\begin{document}

\title[Finite entropy actions of free groups]{Finite entropy actions of free groups, rigidity of stabilizers, and a Howe--Moore type phenomenon}


\author{Brandon Seward}
\address{Department of Mathematics, University of Michigan, 530 Church Street, Ann Arbor, MI 48109, U.S.A.}
\email{b.m.seward@gmail.com}

\begin{abstract}
We study a notion of entropy for probability measure preserving actions of finitely generated free groups, called f-invariant entropy, introduced by Lewis Bowen. In the degenerate case, the f-invariant entropy is negative infinity. In this paper we investigate the qualitative consequences of having finite f-invariant entropy. We find three main properties of such actions. First, the stabilizers occurring in factors of such actions are highly restricted. Specifically, the stabilizer of almost every point must be either trivial or of finite index. Second, such actions are very chaotic in the sense that, when the space is not essentially countable, every non-identity group element acts with infinite Kolmogorov--Sinai entropy. Finally, we show that such actions display behavior reminiscent of the Howe--Moore property. Specifically, if the action is ergodic then there is an integer $n$ such that for every non-trivial normal subgroup $K$ the number of $K$-ergodic components is at most $n$. Our results are based on a new formula for f-invariant entropy.
\end{abstract}
\maketitle

\section{Introduction}

Recently Lewis Bowen \cite{B10a} defined a numerical measure conjugacy invariant for probability measure preserving actions of finitely generated free groups, called f-invariant entropy. The f-invariant entropy is relatively easy to calculate, has strong similarities with the classical Kolmogorov--Sinai entropy of actions of amenable groups, and in fact agrees with the classical Kolmogorov--Sinai entropy when the finitely generated free group is simply $\Z$. Moreover, f-invariant entropy is essentially a special, simpler case of the recently emerging entropy theory of sofic group actions being developed by Bowen (\cite{B10b}, \cite{B10c}, \cite{Ba}), Kerr--Li (\cite{KL}, \cite{KL11a}, \cite{KL11b}), Kerr (\cite{Ke}), and others (\cite{C}, \cite{Z}, \cite{ZC}). The classical Kolmogorov--Sinai entropy has unquestionably been a fundamental and powerful tool in the study of actions of amenable groups, and f-invariant entropy seems posed to take a similar role in the study of actions of finitely generated free groups. There is therefore a significant need to develop and understand the theory of f-invariant entropy. This paper serves as a piece of this large program. We study how f-invariant entropy, or more specifically the property of having finite f-invariant entropy, relates to the qualitative dynamical properties of the action.

Let us define f-invariant entropy. Let $G$ be a finitely generated free group, let $S$ be a free generating set for $G$, and let $G$ act on a standard probability space $(X, \mu)$ by measure preserving bijections. If $\alpha$ is a countable measurable partition of $X$ and $F \subseteq G$ is finite, then we define
$$F \cdot \alpha = \bigvee_{f \in F} f \cdot \alpha.$$
Recall that the \emph{Shannon entropy} of a countable measurable partition $\alpha$ of $X$ is
$$\sH(\alpha) = \sum_{A \in \alpha} -\mu(A) \cdot \log(\mu(A)).$$
Also recall that $\alpha$ is \emph{generating} if the smallest $G$-invariant $\sigma$-algebra containing $\alpha$ contains all measurable sets up to sets of measure zero. If there exists a generating partition $\alpha$ having finite Shannon entropy, then the f-invariant entropy of this action is defined to be
$$f_G(X, \mu) = \lim_{n \rightarrow \infty} F_G(X, \mu, S, \B{n}{S} \cdot \alpha),$$
where
$$F_G(X, \mu, S, \beta) = (1 - 2 r) \cdot \sH(\beta) + \sum_{s \in S} \sH(s \cdot \beta \vee \beta)$$
and $r = |S|$ is the rank of $G$ and $\B{n}{S}$ is the ball of radius $n$ centered on $1_G$ with respect to the generating set $S$. Surprisingly, Bowen proved in \cite{B10a} and \cite{B10c} that the above limit always exists (the terms in the limit are non-increasing) and the value $f_G(X, \mu)$ neither depends on the choice of free generating set $S$ nor on the choice of finite Shannon entropy generating partition $\alpha$. If there is no generating partition for this action having finite Shannon entropy, then the f-invariant entropy is undefined.

Unlike Kolmogorov--Sinai entropy, f-invariant entropy may be negative. In fact for some actions $f_G(X, \mu) = -\infty$. However, one always has $f_G(X, \mu) \leq \sH(\alpha) < \infty$ \cite{B10a}. Thus the conditions that $f_G(X, \mu)$ is finite and $f_G(X, \mu) \neq - \infty$ are equivalent.

The goal of this paper is to expand our knowledge on how f-invariant entropy is related to qualitative dynamical properties. There have previously only been a few results of this type. Bowen proved that Bernoulli shifts over finitely generated free groups are classified up to measure conjugacy by their f-invariant entropy, when it is defined \cite{B10a}, and he proved that actions with negative f-invariant entropy cannot be factors of Bernoulli shifts \cite{B10c}. In \cite[Proof of Lemma 3.5]{BG} Bowen and Gutman showed that for any action $G \acts (X, \mu)$ on an atomless probability space with defined and finite f-invariant entropy, there must be a cyclic subgroup of $G$ for which the induced action of this subgroup has infinite Kolmogorov--Sinai entropy (we strengthen this in Theorem \ref{INTRO KOLM} below). Finally, in \cite[Corollary 1.2]{S12} the author showed that if $G \acts (X, \mu)$ has defined and positive f-invariant entropy and $H \leq G$ is any finitely generated infinite-index subgroup then the restricted action $H \acts (X, \mu)$ does not admit any generating partition having finite Shannon entropy.

Our results will assume very little of the action. We will generally only assume that the f-invariant entropy be defined and be finite, and sometimes we may assume that the measure is ergodic or that it is not supported on a countable set. Our main theorem is below. Before stating this theorem we remind the reader that a point $y$ in a probability space $(Y, \nu)$ is an atom if $\nu(\{y\}) > 0$.

\begin{thm} \label{INTRO STAB}
Let $G$ be a finitely generated non-cyclic free group acting on a standard probability space $(X, \mu)$ by measure preserving bijections. Assume that $f_G(X, \mu)$ is defined. If $f_G(X, \mu) \neq - \infty$ and $(Y, \nu)$ is {\bf any} factor of $(X, \mu)$, then for $\nu$-almost every $y \in Y$, the stabilizer of $y$ is either trivial or has finite index in $G$. Furthermore, $\nu$-almost every $y \in Y$ with non-trivial stabilizer is an atom, and thus there are essentially only countably many points with non-trivial stabilizer.
\end{thm}

We note the following immediate corollary.

\begin{cor} \label{INTRO POSITIVE}
Let $G$ be a finitely generated free group acting on a standard probability space $(X, \mu)$ by measure preserving bijections. Assume that $f_G(X, \mu)$ is defined. If $G \acts (X, \mu)$ is ergodic and $f_G(X, \mu) > 0$ then the action is essentially free.
\end{cor}

We mention that there are examples of actions of finitely generated free groups with defined and finite f-invariant entropy which admit factors whose f-invariant entropy is not defined (see \cite{B11} and \cite{KL11b}). Thus in the above theorem the passage to a factor of $(X, \mu)$ is not superfluous.

The above theorem says that,  ignoring atoms,  all actions with finite f-invariant entropy and all of their factors are essentially free. Since every Bernoulli shift over a finitely generated free group is a factor of a Bernoulli shift with defined and finite f-invariant entropy \cite{B11}, this theorem implies that all non-trivial factors of Bernoulli shifts over finitely generated free groups are essentially free. This fact also follows from a recent result of Robin Tucker-Drob \cite{TD12} which completely characterizes those groups $G$ for which every non-trivial factor of a Bernoulli shift over $G$ is essentially free. Ornstein \cite{Or70} proved that factors of Bernoulli shifts over $\Z$ are again Bernoulli shifts, but it is not known if the same is true for Bernoulli shifts over free groups.

For some groups, the stabilizers which can appear in probability measure preserving actions are quite restricted (such as for higher rank semi-simple groups, by a well known result of Stuck and Zimmer \cite{SZ94}). However, a recent paper by Bowen \cite[Theorem 3.1 and Remark 1]{Bb} shows that probability measure preserving actions of free groups have a ``zoo'' of possible stabilizers. Thus the above theorem demonstrates a significant restriction imposed by having finite f-invariant entropy.

We derive the following corollary which exhibits a remarkable restriction on the ergodic decompositions of induced actions of normal subgroups. This property is somewhat reminiscent of the Howe--Moore property \cite{HM}. A second countable locally compact group is said to have the Howe--Moore property if every ergodic action is mixing (this is not the standard definition but is equivalent; see \cite{CCLTV}). In particular, such actions have the property that every infinite subgroup acts ergodically.

\begin{cor} \label{INTRO SUBERG}
Let $G$ be a finitely generated non-cyclic free group acting on a standard probability space $(X, \mu)$ by measure preserving bijections. Assume that $f_G(X, \mu)$ is defined. If $G \acts (X, \mu)$ is ergodic and $f_G(X, \mu) \neq -\infty$, then there is $n \in \N$ such that for every non-trivial normal subgroup $K \lhd G$ the number of ergodic components of $K \acts (X, \mu)$ is at most $n$.
\end{cor}

In particular, if $\Gamma \leq G$ contains a non-trivial normal subgroup of $G$ then the number of ergodic components of $\Gamma \acts (X, \mu)$ is at most $n$. Notice that the subgroup $\Gamma$ is not required to have finite index in $G$. After proving the above corollary in Section \ref{SEC MAIN}, we will present a construction due to Lewis Bowen which demonstrates that the above corollary cannot in general be extended to hold for all non-trivial subgroups of $G$ (see Proposition \ref{BOWEN EX}).

We prove our main theorem, Theorem \ref{INTRO STAB}, by studying the ergodic components of the action. This requires us to understand how f-invariant entropy behaves with respect to ergodic decompositions. We obtain the following.

\begin{thm} \label{INTRO ERGDEC}
Let $G$ be a finitely generated free group of rank $r$ acting on a standard probability space $(X, \mu)$ by measure preserving bijections. Assume that $f_G(X, \mu)$ is defined. If $\tau$ is the ergodic decomposition of $\mu$ then $f_G(X, \nu)$ is defined for $\tau$-almost every ergodic measure $\nu$ and
$$f_G(X, \mu) = \int f_G(X, \nu) d \tau - (r - 1) \cdot \sH(\tau).$$
\end{thm}

From this we obtain another consequence of having finite f-invariant entropy.

\begin{cor} \label{INTRO ERG}
Let $G$ be a finitely generated non-cyclic free group acting on a standard probability space $(X, \mu)$ by measure preserving bijections. Assume that $f_G(X, \mu)$ is defined. If $f_G(X, \mu) \neq - \infty$ then the action has only countably many ergodic components.
\end{cor}

The final dynamical property we study is the Kolmogorov--Sinai entropy of the restricted actions of the cyclic subgroups of $G$. The following theorem demonstrates that actions having finite f-invariant entropy are quite complicated. Recall that a measure is purely atomic if it gives full measure to a countable set.

\begin{thm} \label{INTRO KOLM}
Let $G$ be a finitely generated non-cyclic free group acting on a standard probability space $(X, \mu)$ by measure preserving bijections. Assume that $f_G(X, \mu)$ is defined. If $\mu$ is not purely atomic and $f_G(X, \mu) \neq - \infty$ then for every $1_G \neq g \in G$ the restricted action of $\langle g \rangle$ on $(X, \mu)$ has infinite Kolmogorov--Sinai entropy.
\end{thm}

For actions of $\Z$ the f-invariant entropy is equal to the Kolmogorov--Sinai entropy when the former is defined. A natural extension of the above theorem would be to determine the f-invariant entropies for the restricted actions of finitely generated subgroups of $G$ (when they are defined). In the case of subgroups of finite index, it was determined by the author in \cite{S12} that the f-invariant entropy is scaled by the index of the subgroup, generalizing a well known property of Kolmogorov--Sinai entropy.

The proof of Theorems \ref{INTRO STAB} and \ref{INTRO KOLM} rely heavily on a new formula for f-invariant entropy. Let us briefly describe this new formula. Fix a free generating set $S$ for $G$ and let $|g|$ denote the reduced $S$-word length of $g \in G$. We define a well ordering $\preceq$ on $G$ as follows. If $|g| < |h|$ then we declare $g \preceq h$. If $|g| = |h|$ then we use a fixed total ordering of $S \cup S^{-1}$ and declare $g \preceq h$ if and only if the reduced $S$-word representation of $g$ lexicographically precedes the reduced $S$-word representation of $h$. Specifically if $|g| = |h| = n$, $g = g_1 g_2 \cdots g_n$, and $h = h_1 h_2 \cdots h_n$ are the reduced $S$-word representations of $g$ and $h$, then $g \preceq h$ if $g = h$ or if $g_i$ is less than $h_i$ for the first $i$ with $g_i \neq h_i$. For $g \in G$ we let $\Pre(g)$ be the set of all group elements which strictly precede $g$. Let $G$ act by measure preserving bijections on a standard probability space $(X, \mu)$. Assume that there is a generating partition $\alpha$ having finite Shannon entropy. Fix $1_G \neq g \in G$. Let $s \in S \cup S^{-1}$ be such that $|s^{-1} g| = |g| - 1$. The \emph{independence decay at $g$ relative to $(S, \alpha)$} is
$$\delta_S(g, \alpha) := \sH(s^{-1} g \cdot \alpha / \Pre(s^{-1} g) \cdot \alpha) - \sH(g \cdot \alpha / \Pre(g) \cdot \alpha).$$
We remark that $\delta_S(g, \alpha) \geq 0$. When $\delta_S(g, \alpha) = 0$, the partitions $g \cdot \alpha$ and $\Pre(g) \cdot \alpha$ are as independent as possible while respecting the fact that $G$ preserves the measure and while keeping $\Pre(s^{-1} g) \cdot \alpha$ fixed. The following new formula for f-invariant entropy is vital to our proofs as it provides much tighter control over the behavior of $f_G(X, \mu)$. We expect this formula to continue to play an important role in the study of f-invariant entropy.

\begin{thm} \label{INTRO DISS}
Let $G$ be a finitely generated free group acting on a standard probability space $(X, \mu)$ by measure preserving bijections. Assume that this action admits a generating partition $\alpha$ having finite Shannon entropy. If $S$ is a free generating set for $G$ then
$$f_G(X, \mu) = \sH(\alpha) - \frac{1}{2} \cdot \sum_{1_G \neq g \in G} \delta_S(g, \alpha).$$
\end{thm}

We will actually prove Theorems \ref{INTRO KOLM} and \ref{INTRO DISS} in the more general context of relative f-invariant entropy and relative Kolmogorov--Sinai entropy of the action relative to a factor action. In Section \ref{SEC KOLM} we answer a question of Lewis Bowen \cite{B10d} by showing that a certain expression is indeed equal to the relative f-invariant entropy (Corollary \ref{COR RFORMULA}).

\subsection*{Organization}
Notation, definitions, and some facts regarding f-invariant entropy are discussed in Section \ref{SEC PRE}. We also deduce Corollary \ref{INTRO POSITIVE} in this section. In Section \ref{SEC DISS} we obtain two new formulas for f-invariant entropy and prove Theorem \ref{INTRO DISS}. We also prove that actions with finite f-invariant entropy cannot factor through a proper quotient of $G$. In Section \ref{SEC KOLM} we apply our new formula to obtain Theorem \ref{INTRO KOLM} and answer a question of Lewis Bowen. We study ergodic decompositions in Section \ref{SEC ERGDEC} and prove Theorem \ref{INTRO ERGDEC} and Corollary \ref{INTRO ERG}. The main theorem of the paper, Theorem \ref{INTRO STAB}, is proved in Section \ref{SEC MAIN}. At the end of Section \ref{SEC MAIN}, we prove Corollary \ref{INTRO SUBERG} and present a construction due to Lewis Bowen (Proposition \ref{BOWEN EX}). 

\subsection*{Acknowledgments}
This material is based upon work supported by the National Science Foundation Graduate Student Research Fellowship under Grant No. DGE 0718128. The author would like to thank his advisor, Ralf Spatzier, for helpful conversations. The author would also like to thank Lewis Bowen for sharing a construction (Proposition \ref{BOWEN EX}) which demonstrates that Corollary \ref{INTRO SUBERG} cannot be strengthened to hold for all non-trivial subgroups.

\section{Preliminaries} \label{SEC PRE}

Throughout this paper $G$ will always denote a finitely generated free group and $S$ will be a free generating set for $G$. The \emph{rank} of $G$ is the minimum size of a generating set for $G$, which in this case is simply $|S|$. We will denote the rank of $G$ by $r$. If $g \in G$ then the \emph{reduced $S$-word representation} of $g$ is the unique (possibly empty) sequence $(s_1, s_2, \ldots, s_n)$ where each $s_i \in S \cup S^{-1}$, $s_i \neq s_{i+1}^{-1}$, and $g = s_1 s_2 \cdots s_n$. The \emph{reduced $S$-word-length} of $g \in G$, denoted $|g|$, is the length of the reduced $S$-word representation of $g$. We let $\B{n}{S} = \{g \in G \: |g| \leq n\}$ be the ball of radius $n$ and $\Sp{n}{S} = \B{n}{S} \setminus \B{n-1}{S}$ be the sphere of radius $n$. We do not emphasize the dependence of $\B{n}{S}$ and $\Sp{n}{S}$ on $S$ as we will never use more than one generating set for $G$ simultaneously. The \emph{left $S$-Cayley graph of $G$} is the graph with vertex set $G$ and edge set $\{(g, s g) \: g \in G, s \in S \cup S^{-1}\}$.

We will use the term \emph{probability space} to always mean a standard Borel space equipped with a Borel probability measure. We will assume that all actions on probability spaces are by measure preserving bijections. If $G$ acts on $(X, \mu)$ then we let $\M(X)$ denote the set of $G$-invariant Borel probability measures on $X$ and we let $\E(X) \subseteq \M(X)$ denote the ergodic measures. When needed, we will write $\M_G(X)$ and $\E_G(X)$ to distinguish the acting group. The set $\M(X)$ is naturally a standard Borel space; its collection of Borel sets is defined to be the smallest $\sigma$-algebra making the maps $\nu \in \M(X) \mapsto \nu(B)$ measurable for every Borel set $B \subseteq X$ \cite[Theorem 17.24]{K95}. If $G$ acts on $(X, \mu)$ then the \emph{ergodic decomposition of $\mu$} is the unique Borel probability measure $\tau$ on $\M(X)$ satisfying $\tau(\E(X)) = 1$ and $\mu = \int \nu d \tau$ (meaning $\mu(B) = \int \nu(B) d \tau$ for every Borel set $B \subseteq X$). If $\pi: (X, \mu) \rightarrow (Y, \nu)$ is a measure-preserving factor map then the \emph{disintegration of $\mu$ with respect to $\nu$} is the unique (up to a $\nu$-null set) collection of probability measures $\{\mu_y \: y \in Y\}$ on $X$ such that $\mu_y(\pi^{-1}(y)) = 1$ for $\nu$-almost-every $y$ and $\mu = \int \mu_y d \nu$ (meaning $\mu(B) = \int \mu_y(B) d \nu$ for every Borel set $B \subseteq X$). 

If $\{\alpha_i \: i \in I\}$ is a finite collection of countable measurable partitions of $X$, then we let $\bigvee_{i \in I} \alpha_i$ denote the coarsest measurable partition of $X$ which is finer than each $\alpha_i$. If $I$ is infinite then we let $\bigvee_{i \in I} \alpha_i$ be the smallest $\sigma$-algebra containing every member of every $\alpha_i$. If $G \acts (X, \mu)$ and $\alpha = \{A_i \: i \in I\}$ is a countable measurable partition of $X$, then we define $g \cdot \alpha = \{g \cdot A_i \: i \in I\}$. For $F \subseteq G$ (finite or infinite) we define $F \cdot \alpha = \bigvee_{f \in F} f \cdot \alpha$. A countable measurable partition $\alpha$ is \emph{generating} for $G \acts (X, \mu)$ if for every Borel set $B \subseteq X$ there is $B' \in G \cdot \alpha$ with $\mu(B \symd B') = 0$.

If $\alpha$ and $\beta$ are countable measurable partitions of $X$ then the \emph{conditional Shannon entropy of $\alpha$ relative to $\beta$} is
$$\sH(\alpha / \beta) = \sum_{B \in \beta} \sum_{A \in \alpha} - \mu(B) \cdot \frac{\mu(B \cap A)}{\mu(B)} \cdot \log \left( \frac{\mu(B \cap A)}{\mu(B)} \right).$$
If $\beta = \{X\}$ is the trivial partition then $\sH(\alpha / \beta)$ equals the Shannon entropy of $\alpha$, $\sH(\alpha)$, as defined in the introduction. When needed we write $\sH_\mu(\alpha / \beta)$ and $\sH_\mu(\alpha)$ to clarify the measure being used. If $\Sigma$ is a sub-$\sigma$-algebra then the \emph{conditional Shannon entropy of $\alpha$ relative to $\Sigma$} is
$$\sH(\alpha / \Sigma) = \int \sH_{\mu_y}(\alpha) d \nu,$$
where $\{\mu_y \: y \in Y\}$ is the disintegration of $\mu$ with respect to the factor map $(X, \mu) \rightarrow (Y, \nu)$ induced by $\Sigma$ (this is not the standard definition, but it is equivalent; see \cite[Section I.1.5]{D}). If $(X, \mu)$ is a probability space then $\sH(\mu)$ is defined as the supremum of $\sH_\mu(\alpha)$ over all finite measurable partitions $\alpha$ of $X$. An \emph{atom} of $\mu$ is a point $x \in X$ with $\mu(\{x\}) > 0$. A probability measure is \emph{purely atomic} if the complement of the set of atoms has measure $0$. Since every probability space (in our sense) is isomorphic to an interval of the real line with Lebesgue measure together with a countable number of atoms, it is easy to show that $\mu$ is purely atomic if $\sH(\mu) < \infty$. The converse does not hold. The following lemma consists of some well known facts on Shannon entropy which we will need. The reader can consult \cite{W82} for a proof.

\begin{lem} \label{LEM SHAN}
Let $(X, \mu)$ be a standard probability space, let $\alpha, \beta, \xi$ be countable measurable partitions of $X$ and let $\Sigma$ be a sub-$\sigma$-algebra. Then
\begin{enumerate}
\item[\rm (i)] $\sH(\alpha / \beta) \geq 0$;
\item[\rm (ii)] $\sH(\alpha \vee \beta) = \sH(\alpha / \beta) + \sH(\beta)$;
\item[\rm (iii)] $\sH(\alpha \vee \beta / \Sigma) = \sH(\alpha / \beta \vee \Sigma) + \sH(\beta / \Sigma)$;
\item[\rm (iv)] $\sH(\alpha / \beta \vee \Sigma) \leq \sH(\alpha / \beta)$;
\item[\rm (v)] $\sH(\alpha \vee \xi / \beta \vee \xi) = \sH(\alpha / \beta \vee \xi)$.
\end{enumerate}
\end{lem}

We will later need the following theorem which places a restriction on the f-invariant entropy of factors.

\begin{thm}[Bowen, \cite{B10c}] \label{BOWEN FACTOR}
Let $G$ be a finitely generated free group acting on a probability space $(X, \mu)$. Assume that this action admits a generating partition $\alpha$ with $\sH(\alpha) < \infty$. If $(Y, \nu)$ is a factor of $(X, \mu)$ and $f_G(Y, \nu)$ is defined, then
$$f_G(Y, \nu) \geq f_G(X, \mu) - \sH(\alpha).$$
\end{thm}

We also observe a simple lemma for later reference.

\begin{lem} \label{LEM FINACT}
Let $G$ be a finitely generated free group of rank $r$ acting on a probability space $(X, \mu)$. If $\sH(\mu) < \infty$ then $f_G(X, \mu) = - (r - 1) \sH(\mu).$
\end{lem}

\begin{proof}
As $\sH(\mu) < \infty$, $\mu$ must be purely atomic. We can therefore partition $X$ so that every atom of $\mu$ is a class of this partition. Call this partition $\alpha$. One readily has $\sH(\alpha) = \sH(\mu) < \infty$ and $\alpha$ is generating. Also, for finite $K \subseteq G$ the partitions $K \cdot \alpha$ and $\alpha$ are, modulo sets of measure zero, identical. Therefore $F_G(X, \mu, S, K \cdot \alpha) = - (r - 1) \sH(\alpha) = - (r - 1) \sH(\mu)$ for every finite $K \subseteq G$.
\end{proof}

We point out that since $- (r - 1) \cdot \sH(\mu) \leq 0$, Theorem \ref{INTRO STAB} and the lemma above easily imply Corollary \ref{INTRO POSITIVE}.

In Sections \ref{SEC DISS} and \ref{SEC KOLM} we will work with relative f-invariant entropy. Specifically, let $G$ act on $(X, \mu)$ and let $\Sigma$ be a $G$-invariant sub-$\sigma$-algebra. If there is a generating partition $\alpha$ having finite Shannon entropy then the \emph{f-invariant entropy of $G \acts (X, \mu)$ relative to $\Sigma$} is
$$f_G(X, \mu / \Sigma) = \lim_{n \rightarrow \infty} F_G(X, \mu / \Sigma, S, \B{n}{S} \cdot \alpha)$$
where
$$F_G(X, \mu / \Sigma, S, \beta) = (1 - 2r) \sH(\beta / \Sigma) + \sum_{s \in S} \sH(s \cdot \beta \vee \beta / \Sigma).$$
In \cite{B10c, B10d} Bowen proved that $f_G(X, \mu / \Sigma)$ neither depends on the choice of $S$ nor $\alpha$. Furthermore, he showed that if the factor $G \acts (Y, \nu)$ induced by $\Sigma$ has defined f-invariant entropy then
$$f_G(X, \mu / \Sigma) = f_G(X, \mu) - f_G(Y, \nu).$$
The relative f-invariant entropy is a generalization of the standard f-invariant entropy, since if $\Sigma = \{\varnothing, X\}$ is the trivial $\sigma$-algebra then $f_G(X, \mu / \Sigma) = f_G(X, \mu)$.

\section{A New Formula} \label{SEC DISS}

In this section and the next we work with relative f-invariant entropy. This does not make the proofs more complicated in any manner whatsoever. Our invariant sub-$\sigma$-algebra will always be denoted $\Sigma$, and if readers wish they can easily either ignore $\Sigma$ (it mostly sits in the background) or take $\Sigma$ as the trivial sub-$\sigma$-algebra $\{X, \varnothing\}$.

In this section we prove that (relative) f-invariant entropy can be computed from what we call \emph{independence decay}. Obtaining this new formula for f-invariant entropy is a key ingredient to many of our proofs.

We begin with a simple lemma which will allow us to simplify the formula for f-invariant entropy.

\begin{lem} \label{LEM SPINEQ}
Let $G$ have rank $r$ and let $G$ act on a probability space $(X, \mu)$. If $\alpha$ is a countable measurable partition of $X$ and $\Sigma$ is a $G$-invariant sub-$\sigma$-algebra, then
$$\sH(\B{n+1}{S} \cdot \alpha / \B{n}{S} \cdot \alpha \vee \Sigma) \leq \sum_{s \in S \cup S^{-1}} \sH(s \B{n}{S} \cdot \alpha / \B{n}{S} \cdot \alpha \vee \Sigma),$$
and
$$\sum_{s \in S \cup S^{-1}} \sH(s \B{n+1}{S} \cdot \alpha / \B{n+1}{S} \cdot \alpha \vee \Sigma) \leq (2r - 1) \cdot \sH(\B{n+1}{S} \cdot \alpha / \B{n}{S} \cdot \alpha \vee \Sigma).$$
\end{lem}

\begin{proof}
We begin with the first inequality. Enumerate $S \cup S^{-1}$ as $t_1, t_2, \ldots, t_{2r}$. Set
$$A_i = \bigcup_{j < i} t_j \B{n}{S}.$$
Since $\B{n+1}{S}$ is the union of the $t_i \B{n}{S}$'s, we have by Lemma \ref{LEM SHAN}
$$\sH(\B{n+1}{S} \cdot \alpha / \B{n}{S} \cdot \alpha \vee \Sigma) = \sum_{i = 1}^{2r} \sH(t_i \B{n}{S} \cdot \alpha / \B{n}{S} \cdot \alpha \vee A_i \cdot \alpha \vee \Sigma)$$
$$\leq \sum_{i = 1}^{2r} \sH(t_i \B{n}{S} \cdot \alpha / \B{n}{S} \cdot \alpha \vee \Sigma) = \sum_{s \in S \cup S^{-1}} \sH(s \B{n}{S} \cdot \alpha / \B{n}{S} \cdot \alpha \vee \Sigma).$$
Thus we have the first inequality.

Now we consider the second inequality. Fix $s \in S \cup S^{-1}$. Set
$$C_s = \left( \bigcup_{s^{-1} \neq t \in S \cup S^{-1}} t \B{n}{S} \right) \setminus \B{n}{S} \subseteq \Sp{n+1}{S}.$$
So $C_s$ is the set of $g \in \Sp{n+1}{S}$ whose reduced $S$-word representations do not begin on the left with $s^{-1}$. Notice that every $g \in \Sp{n+1}{S}$ lies in precisely $(2r - 1)$ many $C_s$'s. Also notice that
$$s \B{n+1}{S} \setminus \B{n+1}{S} = s C_s.$$
Fix a total ordering, $\leq$, of $\Sp{n+1}{S}$. For $g \in \Sp{n+1}{S}$, let $P(g)$ be the set of elements of $\Sp{n+1}{S}$ which strictly precede $g$. For $g \in C_s$ we may not have $P(g) \subseteq C_s$, however we do have
$$s P(g) \subseteq \B{n}{S} \cup s \cdot (P(g) \cap C_s).$$
So by Lemma \ref{LEM SHAN}
$$\sH(s \B{n+1}{S} \cdot \alpha / \B{n+1}{S} \cdot \alpha \vee \Sigma) = \sH(s C_s \cdot \alpha / \B{n+1}{S} \cdot \alpha \vee \Sigma)$$
$$= \sum_{g \in C_s} \sH(s g \cdot \alpha / \B{n+1}{S} \cdot \alpha \vee s (P(g) \cap C_s) \cdot \alpha \vee \Sigma) = \sum_{g \in C_s} \sH(s g \cdot \alpha / \B{n+1}{S} \cdot \alpha \vee s P(g) \cdot \alpha \vee \Sigma)$$
$$\leq \sum_{g \in C_s} \sH(s g \cdot \alpha / s \B{n}{S} \cdot \alpha \vee s P(g) \cdot \alpha \vee \Sigma) = \sum_{g \in C_s} \sH(g \cdot \alpha / \B{n}{S} \cdot \alpha \vee P(g) \cdot \alpha \vee \Sigma).$$
Therefore
$$\sum_{s \in S \cup S^{-1}} \sH(s \B{n+1}{S} \cdot \alpha / \B{n+1}{S} \cdot \alpha \vee \Sigma) \leq \sum_{s \in S \cup S^{-1}} \sum_{g \in C_s} \sH(g \cdot \alpha / \B{n}{S} \cdot \alpha \vee P(g) \cdot \alpha \vee \Sigma)$$
$$= (2r - 1) \sum_{g \in \Sp{n+1}{S}} \sH(g \cdot \alpha / \B{n}{S} \cdot \alpha \vee P(g) \cdot \alpha \vee \Sigma) = (2r - 1) \sH(\B{n+1}{S} \cdot \alpha / \B{n}{S} \cdot \alpha \vee \Sigma).$$
This completes the proof.
\end{proof}

We now obtain a somewhat simpler formula for (relative) f-invariant entropy.

\begin{lem} \label{FINV SPHERE}
Let $G$ have rank $r$ and let $G$ act on a probability space $(X, \mu)$. Assume that there is a generating partition $\alpha$ having finite Shannon entropy. Then for any $G$-invariant sub-$\sigma$-algebra $\Sigma$ we have
$$f_G(X, \mu / \Sigma) = \lim_{n \rightarrow \infty} (1 - r) \cdot \sH(\B{n}{S} \cdot \alpha / \Sigma) + \frac{1}{2} \cdot \sH(\B{n+1}{S} \cdot \alpha / \B{n}{S} \cdot \alpha \vee \Sigma).$$
\end{lem}

\begin{proof}
Define
$$F'_G(X, \mu / \Sigma, S, \alpha, n) = (1 - r) \cdot \sH(\B{n}{S} \cdot \alpha / \Sigma) + \frac{1}{2} \cdot \sH(\B{n+1}{S} \cdot \alpha / \B{n}{S} \cdot \alpha \vee \Sigma),$$
$$f'_G(X, \mu / \Sigma) = \lim_{n \rightarrow \infty} F'_G(X, \mu / \Sigma, S, \alpha, n).$$
Since the action of $G$ preserves measure we have
$$F_G(X, \mu / \Sigma, S, \B{n}{S} \cdot \alpha) = (1 - 2 r) \sH(\B{n}{S} \cdot \alpha / \Sigma) + \sum_{s \in S} \sH(s \B{n}{S} \cdot \alpha \vee \B{n}{S} \cdot \alpha / \Sigma)$$
$$= (1 - 2 r) \sH(\B{n}{S} \cdot \alpha / \Sigma) + \frac{1}{2} \cdot \sum_{s \in S \cup S^{-1}} \sH(s \B{n}{S} \cdot \alpha \vee \B{n}{S} \cdot \alpha / \Sigma)$$
$$= (1 - r) \sH(\B{n}{S} \cdot \alpha / \Sigma) + \frac{1}{2} \cdot \sum_{s \in S \cup S^{-1}} \sH(s \B{n}{S} \cdot \alpha / \B{n}{S} \cdot \alpha \vee \Sigma),$$
where the last equality follows from Lemma \ref{LEM SHAN}. So by the first inequality of Lemma \ref{LEM SPINEQ} we have
$$F'_G(X, \mu / \Sigma, S, \alpha, n) \leq F_G(X, \mu / \Sigma, S, \B{n}{S} \cdot \alpha)$$
for every $n \in \N$. Thus $f'_G(X, \mu / \Sigma) \leq f_G(X, \mu / \Sigma)$.

If $f_G(X, \mu / \Sigma) = -\infty$, then we have $f'_G(X, \mu / \Sigma) = f_G(X, \mu / \Sigma)$ as claimed. So suppose that $f_G(X, \mu / \Sigma) \neq -\infty$. Then we have
$$0 = \lim_{n \rightarrow \infty} 2 \cdot F_G(X, \mu / \Sigma, S, \B{n}{S} \cdot \alpha) - 2 \cdot F_G(X, \mu / \Sigma, S, \B{n+1}{S} \cdot \alpha)$$
$$= \lim_{n \rightarrow \infty} 2(1 - r) \sH(\B{n}{S} \cdot \alpha / \Sigma) - 2(1 - r) \sH(\B{n+1}{S} \cdot \alpha / \Sigma)$$
$$+ \sum_{s \in S \cup S^{-1}} \Big(\sH(s \B{n}{S} \cdot \alpha / \B{n}{S} \cdot \alpha \vee \Sigma) - \sH(s \B{n+1}{S} \cdot \alpha / \B{n+1}{S} \cdot \alpha \vee \Sigma) \Big)$$
$$= \lim_{n \rightarrow \infty} (2r - 2) \cdot \sH(\B{n+1}{S} \cdot \alpha / \B{n}{S} \cdot \alpha \vee \Sigma)$$
$$+ \sum_{s \in S \cup S^{-1}} \Big( \sH(s \B{n}{S} \cdot \alpha / \B{n}{S} \cdot \alpha \vee \Sigma) - \sH(s \B{n+1}{S} \cdot \alpha / \B{n+1}{S} \cdot \alpha \vee \Sigma) \Big)$$
$$= \lim_{n \rightarrow \infty} \sum_{s \in S \cup S^{-1}} \sH(s \B{n}{S} \cdot \alpha / \B{n}{S} \cdot \alpha \vee \Sigma) - \sH(\B{n+1}{S} \cdot \alpha / \B{n}{S} \cdot \alpha \vee \Sigma)$$
$$+ (2r - 1) \cdot \sH(\B{n+1}{S} \cdot \alpha / \B{n}{S} \cdot \alpha \vee \Sigma) - \sum_{s \in S \cup S^{-1}} \sH(s \B{n+1}{S} \cdot \alpha / \B{n+1}{S} \cdot \alpha \vee \Sigma).$$
The expression appearing in the last line and the expression appearing in the second to last line are both non-negative by Lemma \ref{LEM SPINEQ}. Since the limit is $0$, we must have that $f_G(X, \mu / \Sigma) \neq -\infty$ implies
\begin{equation} \label{EQ SPLIM}
\lim_{n \rightarrow \infty} \sum_{s \in S \cup S^{-1}} \sH(s \B{n}{S} \cdot \alpha / \B{n}{S} \cdot \alpha \vee \Sigma) - \sH(\B{n+1}{S} \cdot \alpha / \B{n}{S} \cdot \alpha \vee \Sigma) = 0,
\end{equation}
and
$$\lim_{n \rightarrow \infty} (2r - 1) \cdot \sH(\B{n+1}{S} \cdot \alpha / \B{n}{S} \cdot \alpha \vee \Sigma) - \sum_{s \in S \cup S^{-1}} \sH(s \B{n+1}{S} \cdot \alpha / \B{n+1}{S} \cdot \alpha \vee \Sigma) = 0.$$
In particular, when $f_G(X, \mu / \Sigma) \neq -\infty$ we have
$$\lim_{n \rightarrow \infty} F_G(X, \mu / \Sigma, S, \B{n}{S} \cdot \alpha) - F'_G(X, \mu / \Sigma, S, \alpha, n) = 0$$
by Equation \ref{EQ SPLIM}. Thus $f'_G(X, \mu / \Sigma) = f_G(X, \mu / \Sigma)$ in all cases.
\end{proof}

\begin{cor} \label{COR GROWTH}
Let $G$ be of rank $r > 1$ and let $G$ act on a probability space $(X, \mu)$. Assume that there is a generating partition $\alpha$ having finite Shannon entropy. Let $\Sigma$ be a $G$-invariant sub-$\sigma$-algebra. Let $(Y, \nu)$ be the factor of $(X, \mu)$ obtained from $\Sigma$, and let $\{\mu_y \: y \in Y\}$ be the disintegration of $\mu$ over $\nu$. If $f_G(X, \mu / \Sigma) \neq - \infty$ and $\nu(\{y \in Y \: \mu_y \text{ is not purely atomic}\}) > 0$ then
$$\lim_{n \rightarrow \infty} \sqrt[n]{\sH(\B{n+1}{S} \cdot \alpha / \B{n}{S} \cdot \alpha \vee \Sigma)} = 2r - 1.$$
\end{cor}

Notice that if $\Sigma = \{X, \varnothing\}$ is the trivial $\sigma$-algebra then $Y = \{y\}$ is a singleton and $\mu_y = \mu$. So in this case one only needs to assume that $f_G(X, \mu) \neq - \infty$ and $\mu$ is not purely atomic.

\begin{proof}
Since $\alpha$ is generating, we have $\lim_{n \rightarrow \infty} \sH_{\mu_y}(\B{n}{S} \cdot \alpha) = \sH(\mu_y)$ for $\nu$-almost every $y \in Y$. So
$$\lim_{n \rightarrow \infty} \sH(\B{n}{S} \cdot \alpha / \Sigma) = \lim_{n \rightarrow \infty} \int \sH_{\mu_y}(\B{n}{S} \cdot \alpha) d \nu(y) = \int \sH(\mu_y) d \nu(y)$$
by the Monotone Convergence Theorem. If $\mu_y$ is not purely atomic then $\sH(\mu_y) = \infty$. So our assumptions on the $\mu_y$'s imply that $\sH(\B{n}{S} \cdot \alpha / \Sigma)$ tends to infinity. By Lemma \ref{FINV SPHERE} we have
$$0 = \lim_{n \rightarrow \infty} (1 - r) \cdot \sH(\B{n}{S} \cdot \alpha / \Sigma) + \frac{1}{2} \cdot \sH(\B{n+1}{S} \cdot \alpha / \B{n}{S} \cdot \alpha \vee \Sigma) - f_G(X, \mu / \Sigma).$$
Since $r > 1$ it follows that
$$\lim_{n \rightarrow \infty} \sH(\B{n+1}{S} \cdot \alpha / \B{n}{S} \cdot \alpha \vee \Sigma) = + \infty.$$
In the proof of the previous lemma, specifically Equation \ref{EQ SPLIM}, we showed that when $f_G(X, \mu / \Sigma) \neq - \infty$ we have
$$\lim_{n \rightarrow \infty} \sum_{s \in S \cup S^{-1}} \sH(s \B{n}{S} \cdot \alpha / \B{n}{S} \cdot \alpha \vee \Sigma) - \sH(\B{n+1}{S} \cdot \alpha / \B{n}{S} \cdot \alpha \vee \Sigma) = 0,$$
and
$$\lim_{n \rightarrow \infty} (2r - 1) \cdot \sH(\B{n+1}{S} \cdot \alpha / \B{n}{S} \cdot \alpha \vee \Sigma) - \sum_{s \in S \cup S^{-1}} \sH(s \B{n+1}{S} \cdot \alpha / \B{n+1}{S} \cdot \alpha \vee \Sigma) = 0.$$
From these two equations it follows that
$$\lim_{n \rightarrow \infty} (2r - 1) \cdot \sH(\B{n+1}{S} \cdot \alpha / \B{n}{S} \cdot \alpha \vee \Sigma) - \sH(\B{n+2}{S} \cdot \alpha / \B{n+1}{S} \cdot \alpha \vee \Sigma) = 0.$$
Since $\sH(\B{n+1}{S} \cdot \alpha / \B{n}{S} \cdot \alpha \vee \Sigma)$ tends to infinity, it quickly follows from the above equation that $\sH(\B{n+1}{S} \cdot \alpha / \B{n}{S} \cdot \alpha \vee \Sigma)$ must have exponential growth rate between $2r - 1 - \epsilon$ and $2r - 1 + \epsilon$ for every $\epsilon > 0$.
\end{proof}

The previous corollary allows a simple proof that actions which factor through a proper quotient of $G$ must have f-invariant entropy negative infinity, provided it is defined and the space is not purely atomic. This is an extremely weak version of our main theorem on stabilizers and is itself a new result. We include this corollary because it is quite interesting and its proof is substantially simpler than the proof of Theorem \ref{INTRO STAB}.

\begin{cor} \label{COR NORMAL}
Let $G$ have rank $r > 1$ and let $K \lhd G$ be a non-trivial normal subgroup. Let $G$ act on a probability space $(X, \mu)$. Assume that $f_G(X, \mu)$ is defined. If this action factors through $G / K$ and $\mu$ is not purely atomic then $f_G(X, \mu) = - \infty$.
\end{cor}

\begin{proof}
For clarification, we say the action factors through $G / K$ to mean that $K$ stabilizes every $x \in X$. Towards a contradiction, suppose that $f_G(X, \mu) \neq -\infty$. Let $\alpha$ be a generating partition with $\sH(\alpha) < \infty$. Let $\phi : G \rightarrow G / K$ be the factor map. By \cite[Section B]{GH97}, the exponential growth rate of $G /K$ is strictly less than the exponential growth rate of $G$:
$$\lim_{n \rightarrow \infty} \sqrt[n]{|\phi(\B{n}{S})|} < 2r - 1.$$
The number of distinct partitions among $g \cdot \alpha$ with $g \in \B{n}{S}$ is at most $|\phi(\B{n}{S})|$. So
$$\sH(\B{n}{S} \cdot \alpha / \B{n-1}{S} \cdot \alpha) \leq \sH(\B{n}{S} \cdot \alpha) \leq |\phi(\B{n}{S})| \cdot \sH(\alpha).$$
After taking roots and applying the previous corollary we obtain
$$2r - 1 = \lim_{n \rightarrow \infty} \sqrt[n]{\sH(\B{n}{S} \cdot \alpha / \B{n-1}{S} \cdot \alpha)} \leq \lim_{n \rightarrow \infty} \sqrt[n]{|\phi(\B{n}{S})| \cdot \sH(\alpha)} < 2r - 1,$$
a contradiction.
\end{proof}

The above technique may work in the more general setting where
$$\limsup_{n \rightarrow \infty} \sqrt[n]{|\B{n}{S} \cdot x|}$$
is strictly less than $2r - 1$ for $\mu$-almost every $x \in X$. However, recent work of M. Ab\'{e}rt, Y. Glasner, and B. Vir\'{a}g \cite[Theorem 8]{AGV11}, \cite[Proposition 14]{AGV12} shows that for any finitely generated free group $G \neq \Z$ there exists a measure preserving action of $G$ on a standard probability space $(X, \mu)$ such that the stabilizer of $x$ is non-trivial and the above limit equals $2r - 1$ for $\mu$-almost every $x \in X$. Therefore we cannot rely on the above technique to prove Theorem \ref{INTRO STAB}. Instead, we will obtain tighter control over f-invariant entropy by refining the formula for f-invariant entropy found in Lemma \ref{FINV SPHERE}.

We remark that the result of Ab\'{e}rt--Glasner--Vir\'{a}g discussed above again shows that stabilizers occurring in probability measure preserving actions of free groups can be quite bizarre. Our main theorem therefore demonstrates a significant restriction imposed by having finite f-invariant entropy.

In the next theorem we show that (relative) f-invariant entropy can be defined in terms of independence decay. Recall from the paragraph preceding Theorem \ref{INTRO DISS} that we use a special well-ordering, $\preceq$, on $G$ and we let $\Pre(g)$ denote the set of group elements which strictly precede $g$. In the case of relative f-invariant entropy the independence decay is defined as follows. Let $G \acts (X, \mu)$, let $\Sigma$ be a $G$-invariant sub-$\sigma$-algebra, and suppose that $\alpha$ is a generating partition having finite Shannon entropy. If $g \in G$ is not the identity then we set
$$\delta(g, \alpha / \Sigma) = \sH(s^{-1} g \cdot \alpha / \Pre(s^{-1} g) \cdot \alpha \vee \Sigma) - \sH(g \cdot \alpha / \Pre(g) \cdot \alpha \vee \Sigma),$$
where $s \in S \cup S^{-1}$ is such that $|s^{-1} g| = |g| - 1$. We write $\delta(g, \alpha)$ when $\Sigma = \{X, \varnothing\}$ is the trivial $\sigma$-algebra. Technically $\delta$ depends both on $S$ and on the choice of an ordering of $S \cup S^{-1}$, but we do not emphasize this fact. Notice that if $1_G \neq g \in G$ begins on the left with $s \in S \cup S^{-1}$ in its reduced $S$-word representation then
$$s \cdot \Pre(s^{-1} g) \subseteq \Pre(g).$$
Therefore
$$\delta(g, \alpha / \Sigma) = \sH(s^{-1} g \cdot \alpha / \Pre(s^{-1} g) \cdot \alpha \vee \Sigma) - \sH(g \cdot \alpha / \Pre(g) \cdot \alpha \vee \Sigma)$$
$$= \sH(g \cdot \alpha / s \Pre(s^{-1} g) \cdot \alpha \vee \Sigma) - \sH(g \cdot \alpha / \Pre(g) \cdot \alpha \vee \Sigma) \geq 0,$$
as claimed in the introduction. Furthermore, from the equation above we see that $\delta(g, \alpha / \Sigma)$ measures how much $g \cdot \alpha$ depends on $(\Pre(g) \setminus s \Pre(s^{-1} g)) \cdot \alpha$ when conditioned on $s \Pre(s^{-1} g) \cdot \alpha \vee \Sigma$. So the independence decay at $g$ is $0$ when the partitions $g \cdot \alpha$ and $(\Pre(g) \setminus s \Pre(s^{-1} g)) \cdot \alpha$ are independent when conditioned on $s \Pre(s^{-1} g) \cdot \alpha \vee \Sigma$.

\begin{thm} \label{THM DISS}
Let $G$ act on a probability space $(X, \mu)$ and let $\Sigma$ be a $G$-invariant sub-$\sigma$-algebra. Assume that there is a generating partition $\alpha$ having finite Shannon entropy. Then
$$f_G(X, \mu / \Sigma) = \sH(\alpha / \Sigma) - \frac{1}{2} \cdot \sum_{1_G \neq g \in G} \delta(g, \alpha / \Sigma).$$
\end{thm}

\begin{proof}
Let $r$ denote the rank of $G$. As in the proof of Lemma \ref{FINV SPHERE}, define
$$F'_G(X, \mu / \Sigma, S, \alpha, n) = (1 - r) \cdot \sH(\B{n}{S} \cdot \alpha / \Sigma) + \frac{1}{2} \cdot \sH(\B{n+1}{S} \cdot \alpha / \B{n}{S} \cdot \alpha \vee \Sigma).$$
We have
$$F'_G(X, \mu / \Sigma, S, \alpha, 0) = (1 - r) \sH(\alpha / \Sigma) + \frac{1}{2} \cdot \sH(\B{1}{S} \cdot \alpha / \alpha \vee \Sigma)$$
$$= \sH(\alpha / \Sigma) - \frac{1}{2} \cdot \Big( 2r \cdot \sH(\alpha / \Sigma) - \sH(\B{1}{S} \cdot \alpha / \alpha \vee \Sigma) \Big).$$
The difference between consecutive $F_G'$ terms can be rewritten in a similar manner:
$$2 \cdot F'_G(X, \mu / \Sigma, S, \alpha, n) - 2 \cdot F'_G(X, \mu / \Sigma, S, \alpha, n + 1)$$
$$= 2(1 - r) \sH(\B{n}{S} \cdot \alpha / \Sigma) + \sH(\B{n+1}{S} \cdot \alpha / \B{n}{S} \cdot \alpha \vee \Sigma)$$
$$- 2(1 - r) \sH(\B{n+1}{S} \cdot \alpha / \Sigma) - \sH(\B{n+2}{S} \cdot \alpha / \B{n+1}{S} \cdot \alpha \vee \Sigma)$$
$$= (2r - 1) \sH(\B{n+1}{S} \cdot \alpha / \B{n}{S} \cdot \alpha \vee \Sigma) - \sH(\B{n+2}{S} \cdot \alpha / \B{n+1}{S} \cdot \alpha \vee \Sigma).$$
Putting these together and using Lemma \ref{FINV SPHERE} we have that $f_G(X, \mu / \Sigma)$ equals
$$F'_G(X, \mu / \Sigma, S, \alpha, 0) - \frac{1}{2} \cdot \sum_{n = 1}^\infty \Big( 2 \cdot F'_G(X, \mu / \Sigma, S, \alpha, n - 1) - 2 \cdot F'_G(X, \mu / \Sigma, S, \alpha, n) \Big)$$
$$= \sH(\alpha / \Sigma) - \frac{1}{2} \cdot \Big( 2r \cdot \sH(\alpha / \Sigma) - \sH(\B{1}{S} \cdot \alpha / \alpha \vee \Sigma) \Big)$$
$$- \frac{1}{2} \cdot \sum_{n = 1}^\infty \Big( (2r - 1) \sH(\B{n}{S} \cdot \alpha / \B{n-1}{S} \cdot \alpha \vee \Sigma) - \sH(\B{n+1}{S} \cdot \alpha / \B{n}{S} \cdot \alpha \vee \Sigma) \Big).$$

Let $|g|$ denote the reduced $S$-word length of $g \in G$. Since $|g| < |h|$ implies $g \preceq h$ it follows from Lemma \ref{LEM SHAN} that
$$\sH(\B{n}{S} \cdot \alpha / \B{n-1}{S} \cdot \alpha \vee \Sigma) = \sum_{g \in \Sp{n}{S}} \sH(g \cdot \alpha / \Pre(g) \cdot \alpha \vee \Sigma).$$
Since $|\Sp{1}{S}| = 2r$ we have
\begin{equation} \label{EQ1}
\sum_{g \in \Sp{1}{S}} \delta(g, \alpha / \Sigma) = 2r \cdot \sH(\alpha / \Sigma) - \sH(\B{1}{S} \cdot \alpha / \alpha \vee \Sigma).
\end{equation}
Similarly, for $n \geq 1$ each element of $\Sp{n}{S}$ is adjacent to $2r - 1$ many points of $\Sp{n+1}{S}$ in the left $S$-Cayley graph of $G$. So
$$\sum_{g \in \Sp{n+1}{S}} \delta(g, \alpha / \Sigma) = (2r - 1) \sH(\B{n}{S} \cdot \alpha / \B{n-1}{S} \cdot \alpha \vee \Sigma) - \sH(\B{n+1}{S} \cdot \alpha / \B{n}{S} \cdot \alpha \vee \Sigma).$$
It follows from the equation at the end of the previous paragraph that
$$f_G(X, \mu / \Sigma) = \sH(\alpha / \Sigma) - \frac{1}{2} \cdot \sum_{g \in \Sp{1}{S}} \delta(g, \alpha / \Sigma) - \frac{1}{2} \cdot \sum_{n = 1}^\infty \sum_{g \in \Sp{n+1}{S}} \delta(g, \alpha / \Sigma)$$
\[= \sH(\alpha / \Sigma) - \frac{1}{2} \cdot \sum_{1_G \neq g \in G} \delta(g, \alpha / \Sigma). \qedhere\]
\end{proof}

The above theorem implies that when $f_G(X, \mu) \neq -\infty$, $\delta(g, \alpha)$ tends to $0$ as $|g|$ tends to infinity. We mention that the values $\delta(g, \alpha)$ do not satisfy any monotone properties -- the independence decay can be zero for a very long time and then become positive, in fact it can fluctuate between being positive and zero. If $\delta(h, \alpha) = 0$ for all $h \succeq g$, then $(X, \mu)$ is measurably conjugate to a Markov process, and the Markov partition is $\Pre(g) \cdot \alpha$ (see \cite{B10d}). In fact, since $\Pre(g)$ is left $S$-connected, one can always model $(X, \mu)$ by an action for which the independence decay at $h$ is $0$ for every $h \succeq g$ (this is called a \emph{Markov approximation} to the action, see \cite{S12}). One could take the viewpoint that independence decay measures the error in these Markov approximations. Then the above theorem would say that if $f_G(X, \mu) \neq -\infty$ then these Markov approximations converge rapidly enough to the action so that the errors are summable.

The following simple lemma will have some applications in future sections.

\begin{lem} \label{LEM PATH}
Let $G$ act on a probability space $(X, \mu)$ and let $\Sigma$ be a $G$-invariant sub-$\sigma$-algebra. Assume that there is a generating partition $\alpha$ with $\sH(\alpha) < \infty$. If $p_0, p_1, \ldots, p_n$ is a path in the left $S$-Cayley graph of $G$ with $|p_0| < |p_1| < \cdots < |p_n|$, then
$$\sum_{i = 1}^n \delta(p_i, \alpha / \Sigma) = \sH(p_0 \cdot \alpha / \Pre(p_0) \cdot \alpha \vee \Sigma) - \sH(p_n \cdot \alpha / \Pre(p_n) \cdot \alpha \vee \Sigma).$$
\end{lem}

\begin{proof}
The conditions on the $p_i$ imply that
$$\delta(p_i, \alpha / \Sigma) = \sH(p_{i-1} \cdot \alpha / \Pre(p_{i-1}) \cdot \alpha \vee \Sigma) - \sH(p_i \cdot \alpha / \Pre(p_i) \cdot \alpha \vee \Sigma).$$
It follows that all intermediary terms in the sum cancel.
\end{proof}

\section{Kolmogorov--Sinai Entropy} \label{SEC KOLM}

In this section we show that actions with finite (relative) f-invariant entropy are very complicated in the sense that every cyclic subgroup acts with infinite (relative) Kolmogorov--Sinai entropy.

\begin{lem} \label{LEM SMALLSUM}
Let $G$ act on a probability space $(X, \mu)$ and let $\Sigma$ be a $G$-invariant sub-$\sigma$-algebra. If $f_G(X, \mu / \Sigma)$ is defined and $f_G(X, \mu / \Sigma) \neq - \infty$ then for every generating partition $\alpha$ having finite Shannon entropy
$$\lim_{n \rightarrow \infty} \sum_{g \in G \setminus \B{1}{S}} \delta(g, \B{n}{S} \cdot \alpha / \Sigma) = 0.$$
\end{lem}

\begin{proof}
Let $r = |S|$ be the rank of $G$. By Theorem \ref{THM DISS} and Equation \ref{EQ1}
$$f_G(X, \mu / \Sigma) = \sH(\alpha / \Sigma) - \frac{1}{2} \cdot \sum_{1_G \neq g \in \B{1}{S}} \delta(g, \alpha / \Sigma) - \frac{1}{2} \cdot \sum_{g \in G \setminus \B{1}{S}} \delta(g, \alpha / \Sigma)$$
$$= \sH(\alpha / \Sigma) - \frac{1}{2} \Big(2r \cdot \sH(\alpha / \Sigma) - \sH(\B{1}{S} \cdot \alpha / \alpha \vee \Sigma) \Big) - \frac{1}{2} \cdot \sum_{g \in G \setminus \B{1}{S}} \delta(g, \alpha / \Sigma)$$
$$= (1-r) \sH(\alpha / \Sigma) + \frac{1}{2} \cdot \sH(\B{1}{S} \cdot \alpha / \alpha \vee \Sigma) - \frac{1}{2} \cdot \sum_{g \in G \setminus \B{1}{S}} \delta(g, \alpha / \Sigma).$$
This holds for every generating partition $\alpha$. So by replacing $\alpha$ with $\B{n}{S} \cdot \alpha$ we obtain
$$\sum_{g \in G \setminus \B{1}{S}} \delta(g, \B{n}{S} \cdot \alpha / \Sigma) = 2(1 - r) \sH(\B{n}{S} \cdot \alpha / \Sigma) + \sH(\B{n+1}{S} \cdot \alpha / \B{n}{S} \cdot \alpha \vee \Sigma) - 2 \cdot f_G(X, \mu / \Sigma).$$
Now Lemma \ref{FINV SPHERE} and the assumption $f_G(X, \mu / \Sigma) \neq - \infty$ imply that the above expression approaches $0$ as $n$ approaches infinity.
\end{proof}

\begin{lem} \label{LEM TECH}
Let $G$ act on a probability space $(X, \mu)$ and let $\Sigma$ be a $G$-invariant sub-$\sigma$-algebra. Assume that there is a generating partition $\alpha$ having finite Shannon entropy. If $f_G(X, \mu / \Sigma) \neq -\infty$ then for every $\epsilon > 0$ and every $t \in S \cup S^{-1}$
$$\sH(t \B{n}{S} \cdot \alpha / \Pre(t) \B{n}{S} \cdot \alpha \vee \Sigma) > \sH(\B{n-1}{S} \cdot \alpha / \B{n-2}{S} \cdot \alpha \vee \Sigma) - \epsilon$$
for all but finitely many $n \in \N$.
\end{lem}

\begin{proof}
From Lemma \ref{LEM SHAN} we obtain
$$\sH(t \B{n}{S} \cdot \alpha / \Pre(t) \B{n}{S} \cdot \alpha \vee \Sigma) = \sum_{s \in S \cup S^{-1}} \sH(t s \B{n-1}{S} \cdot \alpha / t \Pre(s) \B{n-1}{S} \cdot \alpha \vee \Pre(t) \B{n}{S} \cdot \alpha \vee \Sigma)$$
$$= \sum_{t^{-1} \neq s \in S \cup S^{-1}} \sH(t s \B{n-1}{S} \cdot \alpha / t \Pre(s) \B{n-1}{S} \cdot \alpha \vee \Pre(t) \B{n}{S} \cdot \alpha \vee \Sigma).$$
Notice that if $t^*$ precedes $t$ in the well ordering of $G$ then $t^* s^*$ precedes $t s$ for every $s, s^* \in S \cup S^{-1}$ with $s \neq t^{-1}$. Therefore for every $s \in S \cup S^{-1}$ with $s \neq t^{-1}$ we have
$$\Pre(t) \cdot (S \cup S^{-1}) \subseteq \Pre(t s)$$
and thus
$$\Pre(t) \B{n}{S} = \Pre(t) (S \cup S^{-1}) \B{n-1}{S} \subseteq \Pre(t s) \B{n-1}{S}.$$
Additionally, if $s^*$ precedes $s$ then $t s^*$ will precede $t s$ provided $s \neq t^{-1}$. Thus for $s \in S \cup S^{-1}$ with $s \neq t^{-1}$ we have $t \Pre(s) \B{n-1}{S} \subseteq \Pre(t s) \B{n-1}{S}$. It follows that
$$\sH(t \B{n}{S} \cdot \alpha / \Pre(t) \B{n}{S} \cdot \alpha \vee \Sigma) \geq \sum_{t^{-1} \neq s \in S \cup S^{-1}} \sH(t s \B{n-1}{S} \cdot \alpha / \Pre(t s) \B{n-1}{S} \cdot \alpha \vee \Sigma).$$
For each $t^{-1} \neq s \in S \cup S^{-1}$ we have
$$\delta(t s, \B{n-1}{S} \cdot \alpha / \Sigma) = \sH(s \B{n-1}{S} \cdot \alpha / \Pre(s) \B{n-1}{S} \cdot \alpha \vee \Sigma) - \sH(t s \B{n-1}{S} \cdot \alpha / \Pre(t s) \B{n-1}{S} \cdot \alpha \vee \Sigma).$$
So we deduce
$$\sH(t \B{n}{S} \cdot \alpha / \Pre(t) \B{n}{S} \cdot \alpha \vee \Sigma) \geq \sum_{t^{-1} \neq s \in S \cup S^{-1}} \sH(s \B{n-1}{S} \cdot \alpha / \Pre(s) \B{n-1}{S} \cdot \alpha \vee \Sigma) - \delta(t s, \B{n-1}{S} \cdot \alpha / \Sigma).$$
By Lemma \ref{LEM SMALLSUM},
$$\sum_{t^{-1} \neq s \in S \cup S^{-1}} \delta(t s, \B{n-1}{S} \cdot \alpha / \Sigma) < \epsilon / (2r)$$
for all but finitely many $n$, where $r = |S|$ is the rank of $G$. So in this case we have
$$\sH(t \B{n}{S} \cdot \alpha / \Pre(t) \B{n}{S} \cdot \alpha \vee \Sigma) \geq - \epsilon / (2r) + \sum_{t^{-1} \neq s \in S \cup S^{-1}} \sH(s \B{n-1}{S} \cdot \alpha / \Pre(s) \B{n-1}{S} \cdot \alpha \vee \Sigma).$$
The entire argument up to this point can be repeated with $n$ replaced with $n-1$ and $t$ replaced by $s$, which implies that $\sH(t \B{n}{S} \cdot \alpha / \Pre(t) \B{n}{S} \cdot \alpha \vee \Sigma)$ is at least
$$- \epsilon / (2r) + \sum_{t^{-1} \neq s \in S \cup S^{-1}} \Bigg( - \epsilon / (2r) + \sum_{s^{-1} \neq u \in S \cup S^{-1}} \sH(u \B{n-2}{S} \cdot \alpha / \Pre(u) \B{n-2}{S} \cdot \alpha \vee \Sigma) \Bigg).$$
Now for every $u \in S \cup S^{-1}$ one can find $s \in S \cup S^{-1}$ with $u \neq s^{-1}$ and $s \neq t^{-1}$. Therefore
$$\sH(t \B{n}{S} \cdot \alpha / \Pre(t) \B{n}{S} \cdot \alpha \vee \Sigma) \geq - 2r \cdot \epsilon / (2r) + \sum_{u \in S \cup S^{-1}} \sH(u \B{n-2}{S} \cdot \alpha / \Pre(u) \B{n-2}{S} \cdot \alpha \vee \Sigma)$$
$$= \sH(\B{n-1}{S} \cdot \alpha / \B{n-2}{S} \cdot \alpha \vee \Sigma) - \epsilon$$
for all but finitely many $n$ (the final equality follows from Lemma \ref{LEM SHAN}).
\end{proof}

Before stating the main theorem of this section, we remind the reader of the definition of (relative) Kolmogorov--Sinai entropy. Let $\Z \acts (X, \mu)$, let $\Sigma$ be a $\Z$-invariant sub-$\sigma$-algebra, and let $\alpha$ be a finite Shannon entropy partition of $X$. Say $\Z = \langle z \rangle$ is generated by $z$. Define
$$\h_\Z(X, \mu / \Sigma, \alpha) = \lim_{k \rightarrow \infty} \sH \left( \alpha \middle/ \bigvee_{m = 1}^k z^{-m} \cdot \alpha \vee \Sigma \right).$$
The terms on the right are decreasing with $k$ (as can be seen from Lemma \ref{LEM SHAN}) and thus the limit exists. The Kolmogorov--Sinai entropy of $\Z \acts (X, \mu)$ relative to $\Sigma$, denoted $\h_\Z(X, \mu / \Sigma)$, is defined to be the supremum of $\h_\Z(X, \mu / \Sigma, \alpha)$ as $\alpha$ ranges over all partitions of $X$ having finite Shannon entropy. In the case $\Sigma = \{X, \varnothing\}$ is trivial, this is called the Kolmogorov--Sinai entropy of the action and is denoted $\h_\Z(X, \mu)$. The importance of relative Kolmogorov--Sinai entropy is that it relates the entropy of an action with the entropy of a factor. Specifically, if $\Z \acts (Y, \nu)$ is the factor induced by $\Sigma$ then
$$\h_\Z(X, \mu) = \h_\Z(Y, \nu) + \h_\Z(X, \mu / \Sigma).$$
The theorem below says that actions with finite f-invariant entropy are complicated in the sense of Kolmogorov--Sinai entropy.

\begin{thm} \label{THM KOLM}
Let $G$ have rank $r > 1$, let $G$ act on a probability space $(X, \mu)$, and let $\Sigma$ be a $G$-invariant sub-$\sigma$-algebra. Assume that $f_G(X, \mu / \Sigma)$ is defined. Let $(Y, \nu)$ be the factor of $(X, \mu)$ obtained from $\Sigma$, and let $\{\mu_y \: y \in Y\}$ be the disintegration of $\mu$ over $\nu$. If $f_G(X, \mu / \Sigma) \neq - \infty$ and $\nu(\{y \in Y \: \mu_y \text{ is not purely atomic}\}) > 0$ then for every $1_G \neq g \in G$ we have
$$\h_{\langle g \rangle} (X, \mu / \Sigma) = \infty.$$
\end{thm}

We remind the reader that when $\Sigma = \{X, \varnothing\}$ is trivial, one only needs to assume that $f_G(X, \mu) \neq -\infty$ and that $\mu$ is not purely atomic.

\begin{proof}
For $g \in G$ let $|g|$ denote $S$-word-length of $g$. Let $\alpha$ be a generating partition having finite Shannon entropy. Fix $1_G \neq g \in G$. It suffices to show that $\lim_{n \rightarrow \infty} \h_{\langle g \rangle} (X, \mu / \Sigma, \B{n}{S} \cdot \alpha) = \infty$. We have
$$\h_{\langle g \rangle} (X, \mu / \Sigma, \B{n}{S} \cdot \alpha) = \lim_{k \rightarrow \infty} \sH \left( \B{n}{S} \cdot \alpha \middle/ \bigvee_{m = 1}^k g^{-m} \B{n}{S} \cdot \alpha \vee \Sigma \right)$$
$$= \lim_{k \rightarrow \infty} \sH \left(g^k \B{n}{S} \cdot \alpha \middle/ \bigvee_{m = 0}^{k-1} g^m \B{n}{S} \cdot \alpha \vee \Sigma \right).$$
Since $g^m$ has shorter $S$-word-length than $g^k$ for each $0 \leq m < k$ we have
$$\bigvee_{m = 0}^{k-1} g^m \B{n}{S} \cdot \alpha \quad \text{is coarser than} \quad \Pre(g^k) \B{n}{S} \cdot \alpha.$$
Therefore
$$\h_{\langle g \rangle} (X, \mu / \Sigma, \B{n}{S} \cdot \alpha) \geq \limsup_{k \rightarrow \infty} \sH(g^k \B{n}{S} \cdot \alpha / \Pre(g^k) \B{n}{S} \vee \Sigma).$$
Let $t \in S \cup S^{-1}$ be the left-most letter in the reduced $S$-word representation of $g$. Fix $k$ and let $p_0, p_1, \ldots, p_\ell$ be a path in the left $S$-Cayley graph of $G$ with $p_0 = t$, $p_\ell = g^k$, and
$$1 = |p_0| < |p_1| < \cdots < |p_\ell|.$$
By applying Lemma \ref{LEM PATH} with respect to the partition $\B{n}{S} \cdot \alpha$ we obtain
$$\sum_{i = 1}^\ell \delta(p_i, \B{n}{S} \cdot \alpha / \Sigma) = \sH(t \B{n}{S} \cdot \alpha / \Pre(t) \B{n}{S} \cdot \alpha \vee \Sigma) - \sH(g^k \B{n}{S} \cdot \alpha / \Pre(g^k) \B{n}{S} \cdot \alpha \vee \Sigma)$$
and hence
$$\sH(g^k \B{n}{S} \cdot \alpha / \Pre(g^k) \B{n}{S} \cdot \alpha \vee \Sigma) \geq \sH(t \B{n}{S} \cdot \alpha / \Pre(t) \B{n}{S} \cdot \alpha \vee \Sigma) - \sum_{u \in G \setminus \B{1}{S}} \delta(u, \B{n}{S} \cdot \alpha / \Sigma).$$
Taking the limit supremum as $k \rightarrow \infty$ we obtain
\begin{equation} \label{EQ2}
\h_{\langle g \rangle} (X, \mu / \Sigma, \B{n}{S} \cdot \alpha) \geq \sH(t \B{n}{S} \cdot \alpha / \Pre(t) \B{n}{S} \cdot \alpha \vee \Sigma) - \sum_{u \in G \setminus \B{1}{S}} \delta(u, \B{n}{S} \cdot \alpha / \Sigma).
\end{equation}
Now by applying Lemmas \ref{LEM SMALLSUM} and \ref{LEM TECH} we have that for all but finitely many $n$
$$\h_{\langle g \rangle} (X, \mu / \Sigma, \B{n}{S} \cdot \alpha) \geq \sH(\B{n-1}{S} \cdot \alpha / \B{n-2}{S} \cdot \alpha \vee \Sigma) - 2 \epsilon.$$
Therefore Corollary \ref{COR GROWTH} implies that $\h_{\langle g \rangle} (X, \mu / \Sigma, \B{n}{S} \cdot \alpha)$ tends to infinity.
\end{proof}

\begin{cor} \label{COR KOLMGROW}
Let $G$ have rank $r > 1$, let $G$ act on a probability space $(X, \mu)$, and let $\Sigma$ be a $G$-invariant sub-$\sigma$-algebra. Assume that $f_G(X, \mu / \Sigma)$ is defined. Let $(Y, \nu)$ be the factor of $(X, \mu)$ obtained from $\Sigma$, and let $\{\mu_y \: y \in Y\}$ be the disintegration of $\mu$ over $\nu$. If $f_G(X, \mu / \Sigma) \neq - \infty$ and $\nu(\{y \in Y \: \mu_y \text{ is not purely atomic}\}) > 0$ then for every $1_G \neq g \in G$ and every finite Shannon entropy generating partition $\alpha$
\[\lim_{n \rightarrow \infty} \sqrt[n]{\h_{\langle g \rangle}(X, \mu / \Sigma, \B{n}{S} \cdot \alpha)} = 2 r - 1. \qedhere\]
\end{cor}

\begin{proof}
By the final inequality of the previous proof, we have that for any $\epsilon > 0$ and all but finitely many $n \in \N$
$$\sH(\B{n-1}{S} \cdot \alpha / \B{n-2}{S} \cdot \alpha \vee \Sigma) - 2 \epsilon \leq \h_{\langle g \rangle} (X, \mu / \Sigma, \B{n}{S} \cdot \alpha) \leq \sH(\B{n}{S} \cdot \alpha / \Sigma) \leq |\B{n}{S}| \cdot \sH(\alpha / \Sigma).$$
By taking $n^\text{th}$ roots of the expressions above and taking the limit as $n$ tends to infinity we obtain (by Corollary \ref{COR GROWTH})
\[2r - 1 \leq \lim_{n \rightarrow \infty} \sqrt[n]{\h_{\langle g \rangle}(X, \mu / \Sigma, \B{n}{S} \cdot \alpha)} \leq 2r - 1. \qedhere\]
\end{proof}

The Corollary below answers a question of Bowen stated in \cite{B10d}. In \cite{B10d} Bowen proved the result below under the additional assumption that $\Sigma = G \cdot \beta$ where $\beta$ has finite Shannon entropy.

\begin{cor} \label{COR RFORMULA}
Let $G$ have rank $r$ and let $G$ act on a probability space $(X, \mu)$ and let $\Sigma$ be a $G$-invariant sub-$\sigma$-algebra. Assume there is a finite Shannon entropy generating partition $\alpha$ for $G \acts (X, \mu)$. Then
$$f_G(X, \mu / \Sigma) = \lim_{n \rightarrow \infty} (1 - r) \sH(\B{n}{S} \cdot \alpha / \Sigma) + \sum_{s \in S} \h_{\langle s \rangle}(X, \mu / \Sigma, \B{n}{S} \cdot \alpha).$$
\end{cor}

\begin{proof}
Note that by Lemma \ref{LEM SHAN} we have
$$f_G(X, \mu / \Sigma) = \lim_{n \rightarrow \infty} (1 - r) \sH(\B{n}{S} \cdot \alpha / \Sigma) + \sum_{s \in S} \sH(s \B{n}{S} \cdot \alpha / \B{n}{S} \cdot \alpha \vee \Sigma).$$
Fix $s \in S$. From the definition of Kolmogorov--Sinai entropy we obtain the inequality
$$\h_{\langle s \rangle}(X, \mu / \Sigma, \B{n}{S} \cdot \alpha) \leq \sH(s \B{n}{S} \cdot \alpha / \B{n}{S} \cdot \alpha \vee \Sigma)$$
for every $n \in \N$. Therefore
$$f_G(X, \mu / \Sigma) \geq \lim_{n \rightarrow \infty} (1 - r) \sH(\B{n}{S} \cdot \alpha / \Sigma) + \sum_{s \in S} \h_{\langle s \rangle}(X, \mu / \Sigma, \B{n}{S} \cdot \alpha).$$
In particular, when $f_G(X, \mu / \Sigma) = -\infty$ we have equality above. So now assume that $f_G(X, \mu / \Sigma) \neq -\infty$. Order $S \cup S^{-1}$ so that $s$ is least, and let $\preceq$ be the induced well-ordering of $G$ (see the paragraph preceding Theorem \ref{INTRO DISS}). So $\Pre(s) = \{1_G\}$. We will use Equation \ref{EQ2} from the proof of Theorem \ref{THM KOLM}. Notice that we obtained Equation \ref{EQ2} without using the assumption from Theorem \ref{THM KOLM} that some measures $\mu_y$ are not purely atomic. So this equation holds in our current setting. Recall that in that equation $t$ was the left-most letter in the reduced $S$-word representation of $g$. Using $g = s$ we have $t = s$. Let $\epsilon > 0$. By applying Lemma \ref{LEM SMALLSUM} and using Equation \ref{EQ2} with $g = s$ we obtain
$$\sH(s \B{n}{S} \cdot \alpha / \B{n}{S} \cdot \alpha \vee \Sigma) - \epsilon \leq \h_{\langle s \rangle}(X, \mu / \Sigma, \B{n}{S} \cdot \alpha)$$
for all but finitely many $n \in \N$. This holds for every $s \in S$, so we deduce
\[f_G(X, \mu / \Sigma) - r \cdot \epsilon \leq \lim_{n \rightarrow \infty} (1 - r) \sH(\B{n}{S} \cdot \alpha / \Sigma) + \sum_{s \in S} \h_{\langle s \rangle}(X, \mu / \Sigma, \B{n}{S} \cdot \alpha). \qedhere\]
\end{proof}

\section{Ergodic Decompositions} \label{SEC ERGDEC}

For the remainder of the paper we work with non-relative f-invariant entropy. In this section we relate the f-invariant entropy of an invariant measure to the f-invariant entropy of the ergodic components of the invariant measure. We first show that f-invariant entropy is defined on the ergodic components, however we will find it necessary to work in a slightly more general setting.

\begin{defn}
Let $G$ act on a probability space $(X, \mu)$. A \emph{decomposition} of $\mu$ is a Borel probability measure $\tau$ on $\M(X)$ such that $\mu = \int_{m \in \M(X)} m d \tau$. A decomposition is \emph{countable} if it is purely atomic. We say that $\tau$ is a decomposition of $\mu$ into \emph{mutually singular measures} if $\tau \times \tau$-almost every pair of measures $(m, \lambda) \in \M(X) \times \M(X)$ are either identical ($m = \lambda$) or  are mutually singular, meaning there is a Borel set $B \subseteq X$ with $m(B) = \lambda(X \setminus B) = 1$.
\end{defn}

It is well known that distinct ergodic measures are mutually singular. So ergodic decompositions are decompositions into mutually singular measures.

\begin{lem} \label{LEM GENERATOR}
Let $G$ act on a probability space $(X, \mu)$. Let $\tau$ be a decomposition of $\mu$ into mutually singular measures. If $\alpha$ is a countable measurable partition of $X$ then:
\begin{enumerate}
\item[\rm (i)] if $\alpha$ is generating for $G \acts (X, \mu)$ then $\alpha$ is generating for $G \acts (X, \nu)$ for $\tau$-almost every $\nu \in \M(X)$;
\item[\rm (ii)] if $\sH_\mu(\alpha) < \infty$ then $\sH_\nu(\alpha) < \infty$ for $\tau$-almost every $\nu \in \M(X)$. Furthermore $\int \sH_\nu(\alpha) d \tau \leq \sH_\mu(\alpha)$.
\end{enumerate}
\end{lem}

\begin{proof}
(i). Since $X$ is by assumption a standard Borel space, there is a countable collection of Borel sets $\C$ such that the $\sigma$-algebra generated by $\C$ is precisely the collection of all Borel subsets of $X$. Since $\alpha$ is generating for $G \acts (X, \mu)$, we have that for every $C \in \C$ there is a paired set $p(C) \in G \cdot \alpha$ with $\mu(C \symd p(C)) = 0$. We have
$$0 = \mu(C \symd p(C)) = \int_{\nu \in \M(X)} \nu(C \symd p(C)) d \tau.$$
So $\nu(C \symd p(C)) = 0$ for $\tau$-almost every $\nu \in \M(X)$. Since $\C$ is countable, we can find a single set $M \subseteq \M(X)$ with $\tau(M) = 1$ for which $\nu(C \symd p(C)) = 0$ for every $\nu \in M$ and every $C \in \C$. Now let $\nu \in M$ and let $\mathcal{F}$ be the collection of Borel sets $B \subseteq X$ for which there is $B' \in G \cdot \alpha$ with $\nu(B \symd B') = 0$. It is easily verified that $\mathcal{F}$ is a $\sigma$-algebra, and by definition of $M$ we have $\mathcal{F}$ contains $\C$. Therefore $\mathcal{F}$ is the collection of Borel sets and $\alpha$ is a generating partition for $G \acts (X, \nu)$.

(ii). From the definition of the Borel structure on $\M(X)$ (see Section \ref{SEC PRE}), it is apparent that the functions $\nu \mapsto -\nu(A) \log(\nu(A))$ are Borel for $A \subseteq X$ Borel. Therefore $\nu \mapsto \sH_\nu(\alpha)$ is a Borel function on $\M(X)$. So the integral $\int \sH_\nu(\alpha) d \tau$ is defined and exists since $\sH_\nu(\alpha) \geq 0$ for all $\nu$. By the Monotone Convergence Theorem we have
$$\int_{\nu \in \M(X)} \sH_\nu(\alpha) d \tau = \int_{\nu \in \M(X)} \sum_{A \in \alpha} -\nu(A) \cdot \log(\nu(A)) d \tau$$
$$= \sum_{A \in \alpha} \int_{\nu \in \M(X)} -\nu(A) \cdot \log(\nu(A)) d \tau.$$
Define $\phi: [0, 1] \rightarrow \R$ by $\phi(0) = 0$ and $\phi(t) = t \cdot \log(t)$ for $0 < t \leq 1$. One can easily check that the function $\phi$ is convex. Since the function $\nu \mapsto \nu(A)$ clearly lies in $\mathcal{L}^1(\M(X), \tau)$, we can apply Jensen's Inequality to obtain
$$\int_{\nu \in \M(X)} \sH_\nu(\alpha) d \tau = \sum_{A \in \alpha} \int_{\nu \in \M(X)} - \phi(\nu(A)) d \tau$$
$$\leq \sum_{A \in \alpha} - \phi \left( \int_{\nu \in \M(X)} \nu(A) d \tau \right) = \sum_{A \in \alpha} - \phi(\mu(A)) = \sH_\mu(\alpha).$$
We conclude that if $\sH_\mu(\alpha) < \infty$ then $\sH_\nu(\alpha) < \infty$ for $\tau$-almost every $\nu \in \M(X)$.
\end{proof}

\begin{thm} \label{THM ERGDEC}
Let $G$ have rank $r$ and let $G$ act on a probability space $(X, \mu)$. Assume that $f_G(X, \mu)$ is defined. If $\tau$ is the ergodic decomposition of $\mu$, then $f_G(X, \nu)$ is defined for $\tau$-almost every $\nu \in \E(X)$ and
$$f_G(X, \mu) = \int_{\nu \in \E(X)} f_G(X, \nu) d \tau - (r - 1) \cdot \sH(\tau).$$
\end{thm}

\begin{proof}
This is well known when $G = \Z$ (\cite[Theorem 8.4]{W82}). So we assume below that $r > 1$.

We first treat a special case. Let $\tau$ be a countable decomposition of $\mu$ (not necessarily the ergodic decomposition) into mutually singular measures with $\sH(\tau) < \infty$. Say the atoms of $\tau$ are $(\nu_i)_{i \in I}$ for some countable set $I$. Set $p_i = \tau(\{\nu_i\}) > 0$. Since the $\nu_i$'s are countable and mutually singular, we can find a measurable partition $\xi = \{X_1, X_2, \ldots\}$ of $X$ such that $\nu_i(X_i) = 1$ for each $i \in I$.

If $\beta$ is any countable measurable partition of $X$ then
$$\sH_\mu(\beta / \xi) = \sum_{i \in I} \sum_{B \in \beta} - \mu(X_i) \cdot \frac{\mu(X_i \cap B)}{\mu(X_i)} \cdot \log \left( \frac{\mu(X_i \cap B)}{\mu(X_i)} \right)$$
$$= \sum_{i \in I} \sum_{B \in \beta} - p_i \cdot \nu_i(B) \cdot \log(\nu_i(B)) = \sum_{i \in I} p_i \cdot \sH_{\nu_i}(\beta).$$
In particular, if $\beta$ refines $\xi$ we have
$$\sH_\mu(\beta) = \sH_\mu(\xi) + \sH_\mu(\beta / \xi) = \sH(\tau) + \sum_{i \in I} p_i \cdot \sH_{\nu_i}(\beta),$$
and therefore (assuming $\sH(\beta) < \infty$)
$$F_G(X, \mu, S, \beta) = (1 - 2 r) \sH_\mu(\beta) + \sum_{s \in S} \sH_\mu(s \beta \vee \beta)$$
$$= (1 - r) \sH(\tau) + \sum_{i \in I} p_i \cdot \left( (1 - 2r) \sH_{\nu_i}(\beta) + \sum_{s \in S} \sH_{\nu_i}(s \cdot \beta \vee \beta) \right)$$
$$= (1 - r) \sH(\tau) + \sum_{i \in I} p_i \cdot F_G(X, \nu_i, S, \beta).$$
Let $\alpha$ be a generating partition for $G \acts (X, \mu)$ with $\sH_\mu(\alpha) < \infty$. As $\alpha \vee \xi$ is generating and $\sH_\mu(\alpha \vee \xi) \leq \sH_\mu(\xi) + \sH_\mu(\alpha) = \sH(\tau) + \sH_\mu(\alpha) < \infty$, we can apply Lemma \ref{LEM GENERATOR} to obtain
$$f_G(X, \mu) = \lim_{n \rightarrow \infty} F_G(X, \mu, S, \B{n}{S} \cdot (\alpha \vee \xi))$$
$$= \lim_{n \rightarrow \infty} (1- r) \sH(\tau) + \sum_{i \in I} p_i \cdot F_G(X, \nu_i, S, \B{n}{S} \cdot (\alpha \vee \xi))$$
$$= (1 - r) \sH(\tau) + \sum_{i \in I} p_i \cdot f_G(X, \nu_i) = \int_{\nu \in \M(X)} f_G(X, \nu) d \tau - (r - 1) \sH(\tau).$$
This formula holds whenever $\tau$ is a countable decomposition of $\mu$ into mutually singular measures with $\sH(\tau) < \infty$.

Now let $\tau$ be the ergodic decomposition of $\mu$. If $\sH(\tau) < \infty$ then $\tau$ is countable and by the previous paragraph
$$f_G(X, \mu) = \int_{\nu \in \E(X)} f_G(X, \nu) d \tau - (r - 1) \sH(\tau).$$
Now suppose that $\sH(\tau) = \infty$. Then there is a sequence $(\Lambda_n)_{n \in \N}$ of finite measurable partitions of $\M(X)$ with $\sH_\tau(\Lambda_n)$ tending to infinity as $n$ tends to infinity. Fix $n \in \N$ and for $\lambda \in \Lambda_n$ define
$$m_\lambda = \frac{1}{\tau(\lambda)} \cdot \int_{\nu \in \lambda} \nu d \tau \in \M(X).$$
Let $\zeta$ be the probability measure on $\M(X)$ which has atoms $\{m_\lambda \: \lambda \in \Lambda_n\}$ and satisfies $\zeta(\{m_\lambda\}) = \tau(\lambda)$. Then $\zeta$ is a decomposition of $\mu$ into mutually singular measures and $\sH(\zeta) = \sH_\tau(\Lambda_n) < \infty$. By the paragraph above we have
$$f_G(X, \mu) = \int_{m \in \M(X)} f_G(X, m) d \zeta - (r - 1) \cdot \sH(\zeta) \leq \int_{m \in \M(X)} \sH_m(\alpha) d \zeta - (r - 1) \sH(\zeta)$$
$$\leq \sH_\mu(\alpha) - (r - 1) \sH(\zeta) = \sH_\mu(\alpha) - (r - 1) \cdot \sH_\tau(\Lambda_n).$$
Taking the limit as $n$ tends to infinity we obtain $f_G(X, \mu) = - \infty$. We also have
$$\int_{\nu \in \E(X)} f_G(X, \nu) d \tau - (r - 1) \sH(\tau) \leq \int_{\nu \in \E(X)} \sH_\nu(\alpha) d \tau - (r - 1) \sH(\tau)$$
$$\leq \sH_\mu(\alpha) - (r - 1) \sH(\tau) = - \infty,$$
provided the function $\nu \mapsto f_G(X, \nu)$ is Borel, so that the integral is defined (this is only a concern now because $\tau$ may not be countable). We will resolve this technicality in the next paragraph and thus we will have
$$f_G(X, \mu) = \int_{\nu \in \E(X)} f_G(X, \nu) d \tau - (r - 1) \sH(\tau)$$
in all cases.

As we saw in the proof of Lemma \ref{LEM GENERATOR}, the map $\nu \mapsto \sH_\nu(\beta)$ is Borel for every countable measurable partition $\beta$ of $X$. When $\sH_\nu(\beta) < \infty$ we have that $F_G(X, \nu, S, \beta)$ is defined. By Lemma \ref{LEM GENERATOR}, there is a Borel set $E \subseteq \E(X)$ such that $\tau(E) = 1$ and $\sH_\nu(\B{n}{S} \cdot \alpha) \leq |\B{n}{S}| \cdot \sH_\nu(\alpha) < \infty$ for all $\nu \in E$ and all $n \in \N$. It readily follows that $\nu \mapsto F_G(X, \nu, S, \B{n}{S} \cdot \alpha)$ is a Borel function on $E$. After taking limits, we find that $\nu \mapsto f_G(X, \nu)$ is a Borel function on $E$. This completes the proof.
\end{proof}

Since $\sH(\tau) < \infty$ implies that $\tau$ is purely atomic, the above proof demonstrates the following.

\begin{cor} \label{COR ERGDEC}
Let $G$ have rank $r > 1$ and let $G$ act on a probability space $(X, \mu)$. Assume that $f_G(X, \mu)$ is defined. If $f_G(X, \mu) \neq -\infty$ then the action only has countably many ergodic components.
\end{cor}

\section{Stabilizers of Factor Actions} \label{SEC MAIN}

In this section we prove our main theorem which characterizes the stabilizers which can appear in factors of actions having finite f-invariant entropy.

If $(X, \mu)$ is a probability space and $B \subseteq X$ is Borel with $\mu(B) > 0$, then we define a probability measure $\mu_B$ on $X$ by
$$\mu_B(A) = \frac{\mu(A \cap B)}{\mu(B)}$$
for Borel sets $A \subseteq X$. Notice that if $\alpha$, $\beta$, and $\xi$ are countable measurable partitions of $(X, \mu)$ then
$$\sH(\alpha / \beta \vee \xi) = \sum_{B \in \beta} \sum_{C \in \xi} \sum_{A \in \alpha} - \mu(B \cap C) \cdot \frac{\mu(A \cap B \cap C)}{\mu(B \cap C)} \cdot \log \left( \frac{\mu(A \cap B \cap C)}{\mu(B \cap C)} \right)$$
$$= \sum_{B \in \beta} \mu(B) \cdot \sH_{\mu_B}(\alpha / \xi).$$
Also notice that if $\nu = \mu_B$ then $\nu_C = \mu_{B \cap C}$.

\begin{lem} \label{LEM CONSHAN}
Let $(X, \mu)$ be a probability space. If $\delta > 0$ then there is $\rho > 0$ such that if two countable measurable partitions $\alpha$, $\beta$ of $X$ satisfy $\sH_\mu(\alpha / \beta) < \rho$ then there is a subcollection $\beta' \subseteq \beta$ consisting of positive measure sets such that $\mu(\cup \beta') > 1 - \delta$ and for every $B \in \beta'$ there is $A \in \alpha$ with $\mu_B(A) > 1 - \delta.$
\end{lem}

\begin{proof}
By making $\delta$ smaller if necessary, we may suppose that $\delta < \frac{1}{3}$. Set
$$\epsilon = - \delta \cdot \log(\delta) - (1 - \delta) \cdot \log(1 - \delta) > 0$$
and set $\rho = \epsilon \cdot \delta$. Let $\alpha$ and $\beta$ be countable measurable partitions of $X$ satisfying $\sH_\mu(\alpha / \beta) < \rho$. Define
$$\beta' = \{B \in \beta \: \mu(B) > 0 \text{ and } \sH_{\mu_B}(\alpha) < \epsilon\}.$$
Then
$$\epsilon \cdot \mu(X \setminus \cup \beta') \leq \sum_{B \in \beta} \mu(B) \cdot \sH_{\mu_B}(\alpha) = \sH_\mu(\alpha / \beta) < \rho.$$
So $\mu(X \setminus \cup \beta') < \rho / \epsilon = \delta$. Thus $\mu(\cup \beta') > 1 - \delta$ as required.

Fix $B \in \beta'$. We must show that there is $A \in \alpha$ with $\mu_B(A) > 1 - \delta$. If $A \in \alpha$ satisfies
$$\delta < \mu_B(A) < 1 - \delta,$$
then the partition $\xi = \{A, X \setminus A\}$ is a coarsening of $\alpha$ and we have
$$- \delta \cdot \log(\delta) - (1 - \delta) \cdot \log(1 - \delta) < \sH_{\mu_B}(\xi) \leq \sH_{\mu_B}(\alpha) < \epsilon,$$
contradicting the definition of $\epsilon$. So for every $A \in \alpha$, we have that $\mu_B(A)$ is either less than $\delta$ or greater than $1 - \delta$. Towards a contradiction, suppose that $\mu_B(A) < \delta$ for all $A \in \alpha$. Then we can find a set $C$ which is a union of members of $\alpha$ such that $\mu_B(C) \in [\frac{1}{2} - \frac{\delta}{2}, \frac{1}{2} + \frac{\delta}{2}]$. Then the partition $\gamma = \{C, X \setminus C\}$ is a coarsening of $\alpha$ so
$$- \left( \frac{1}{2} - \frac{\delta}{2} \right) \cdot \log \left( \frac{1}{2} - \frac{\delta}{2} \right) - \left( \frac{1}{2} + \frac{\delta}{2} \right) \cdot \log \left( \frac{1}{2} + \frac{\delta}{2} \right) \leq \sH_{\mu_B}(\gamma)$$
$$\leq \sH_{\mu_B}(\alpha) < \epsilon = - \delta \cdot \log(\delta) - (1 - \delta) \cdot \log(1 - \delta),$$
contradicting the fact that $\delta < \frac{1}{3}$. We conclude there is $A \in \alpha$ with $\mu_B(A) > 1 - \delta$.
\end{proof}

\begin{lem}
Let $G$ act on a probability space $(X, \mu)$. Let $|g|$ denote the reduced $S$-word length of $g \in G$. Suppose that $\mu$-almost every $x \in X$ has non-trivial stabilizer. Then for every $\epsilon > 0$ there exists $1_G \neq g \in G$ for which
$$\mu(\{x \in X \: g \cdot x \in \B{|g|-1}{S} \cdot x\}) > 1 - \epsilon.$$
\end{lem}

\begin{proof}
Fix a total ordering of $S \cup S^{-1}$ and let $\preceq$ be the induced well ordering of $G$ (see the paragraph just before Theorem \ref{INTRO DISS}). For $x \in X$ we define $\phi(x)$ to be the $\preceq$-least $1_G \neq g \in G$ with $g \cdot x = x$. If $x$ has trivial stabilizer then we define $\phi(x) = 1_G$. Then $\phi: X \rightarrow G$ is a measurable function satisfying $\mu(\phi^{-1}(1_G)) = 0$ and $\phi(x) \cdot x = x$ for every $x \in X$.

Fix $\epsilon > 0$ and set $\delta = \frac{\epsilon}{2}$. Let $k \in \N$ be such that $\mu(\phi^{-1}(\B{k}{S})) > 1 - \delta$. Let $M \in \N$ be such that
$$(1 - \delta) \cdot \left( 1 - \left( 1 - \frac{1}{|\B{k}{S}|} \right)^M \right) > 1 - 2 \delta.$$

Set $Y_0 = \varnothing$. We will inductively define $g_j \in G$ and $Y_j \subseteq X$ for $0 < j \leq M$ which satisfy the following two conditions for every $j > 0$:
$$\forall y \in Y_j \ g_j \cdot y \in \B{|g_j|-1}{S} \cdot y;$$
$$\mu(Y_j) > \left(1 - \frac{1}{|\B{k}{S}|} \right) \mu(Y_{j-1}) + \frac{1}{|\B{k}{S}|} \cdot (1 - \delta).$$
We first define $g_1$ and $Y_1$. Since $\mu(\phi^{-1}(\B{k}{S})) > 1 - \delta$, we can pick $1_G \neq g_1 \in \B{k}{S}$ so that $Y_1 = \phi^{-1}(g_1)$ satisfies
$$\mu(Y_1) > \frac{1}{|\B{k}{S}|} ( 1 - \delta).$$
The two properties above are then satisfied since $g_1$ stabilizes every point in $Y_1$. Now suppose that $1 < m \leq M$ and $g_{m-1}$ and $Y_{m-1}$ have been defined and satisfy the two conditions above. Fix any $t \not\in \B{k}{S}$ such that $|t g_{m-1}| = |t| + |g_{m-1}|$. Since $\mu(\phi^{-1}(\B{k}{S})) > 1 - \delta$, we have that
$$\mu \bigg( \phi^{-1}(\B{k}{S}) \setminus t g_{m-1} \cdot Y_{m-1} \bigg) > 1 - \delta - \mu(Y_{m-1}).$$
So we can find $1_G \neq u \in \B{k}{S}$ such that $Z = \phi^{-1}(u)$ satisfies
$$\mu \bigg( Z \setminus t g_{m-1} \cdot Y_{m-1} \bigg) > \frac{1}{|\B{k}{S}|} \cdot \bigg( 1 - \delta - \mu(Y_{m-1}) \bigg).$$

For $h \in G$ let $W_S(h)$ denote the reduced $S$-word representation of $h$. If $W_S(t)$ begins (on the left) with $W_S(u t)$ then $W_S(t^{-1} u^{-1} t)$ and $W_S(t)$ end with the same letter (note that $t^{-1} u^{-1} t \neq 1_G$). In this case we set $p = t^{-1} u^{-1} t$ and observe that
$$|p g_{m-1}| = |p| + |g_{m-1}|.$$
If $W_S(t)$ does not begin with $W_S(u t)$ then $W_S(t^{-1} u t)$ and $W_S(u t)$ end with the same letter. Furthermore, from $|u| < |t|$ we find that $W_S(u t)$ and $W_S(t)$ end with the same letter. So in this case we set $p = t^{-1} u t$ and again observe that
$$|p g_{m-1}| = |p| + |g_{m-1}|.$$

We set $g_m = p g_{m-1}$ and $Y_m = Y_{m-1} \cup (t g_{m-1})^{-1} \cdot Z$. We now check that the two inductive hypothesis are satisfied. Fix $y \in Y_m$. First suppose that $y \in Y_{m-1}$. Then
$$g_m \cdot y = p g_{m-1} \cdot y \in p \B{|g_{m-1}|-1}{S} \cdot y \subseteq \B{|p|+|g_{m-1}|-1}{S} \cdot y = \B{|g_m|-1}{S} \cdot y.$$
Now suppose that $y \in Y_m \setminus Y_{m-1}$. Set $z = t g_{m-1} \cdot y$. Then $z \in Z = \phi^{-1}(u)$. It follows
$$g_m \cdot y = p g_{m-1} \cdot y = t^{-1} u^{\pm 1} t g_{m-1} \cdot y = t^{-1} u^{\pm 1} \cdot z$$
$$= t^{-1} \cdot z = g_{m-1} \cdot y \in \B{|g_{m-1}|}{S} \cdot y \subseteq \B{|g_m|-1}{S} \cdot y.$$
So the first condition is satisfied. For the second, we have
$$\mu(Y_m) = \mu(Y_{m-1}) + \mu(Y_m \setminus Y_{m-1}) = \mu(Y_{m-1}) + \mu(Z \setminus t g_{m-1} \cdot Y_{m-1})$$
$$> \mu(Y_{m-1}) + \frac{1}{|\B{k}{S}|} \cdot (1 - \delta - \mu(Y_{m-1})) = \left(1 - \frac{1}{|\B{k}{S}|} \right) \mu(Y_{m-1}) + \frac{1}{|\B{k}{S}|} \cdot (1 - \delta).$$
This completes the inductive construction.

By our condition on the measures of the $Y_j$'s we have
$$\mu(Y_M) > \frac{1}{|\B{k}{S}|} (1 - \delta) \cdot \sum_{i = 0}^{M-1} \left(1 - \frac{1}{|\B{k}{S}|} \right)^i$$
$$= (1 - \delta) \cdot \left(1 - \left(1 - \frac{1}{|\B{k}{S}|} \right)^M \right) > 1 - 2 \delta = 1 - \epsilon.$$
Using $g = g_M$ completes the proof.
\end{proof}

We now prove a weakened version of the main theorem.

\begin{prop} \label{MAIN PROP}
Let $G$ have rank $r > 1$ and let $G$ act on a probability space $(X, \mu)$. Assume that $f_G(X, \mu)$ is defined. If $G \acts (X, \mu)$ is ergodic and $f_G(X, \mu) \neq - \infty$ then either the action is essentially free or else $\mu$ is purely atomic.
\end{prop}

\begin{proof}
Let $\alpha$ be a generating partition with $\sH(\alpha) < \infty$. Let $r = |S|$ be the rank of $G$, and let $|g|$ denote the reduced $S$-word length of $g \in G$. Also fix a total ordering of $S \cup S^{-1}$, let $\preceq$ be the induced well ordering of $G$ (see the paragraph before Theorem \ref{INTRO DISS}), and let $\Pre(g)$ denote the set of group elements strictly preceding $g$. Assume that $G \acts (X, \mu)$ is ergodic and not essentially free and that $\mu$ is not purely atomic. We will show that $f_G(X, \mu) = -\infty$.

If $\sH(\B{n}{S} \cdot \alpha / \B{n-1}{S} \cdot \alpha)$ does not tend to infinity then $f_G(X, \mu) = -\infty$ by Corollary \ref{COR GROWTH} and we are done. So suppose that $\sH(\B{n}{S} \cdot \alpha / \B{n-1}{S} \cdot \alpha)$ tends to infinity. Fix $0 < \epsilon < 1/2$ and let $n \in \N$ be such that
$$\sH(\B{n}{S} \cdot \alpha / \B{n-1}{S} \cdot \alpha) > \frac{1}{\epsilon}.$$
Set $\delta = \epsilon / (6 \cdot |\Sp{n}{S}|)$. By combining finitely many classes of $\alpha$ into a single class, we can obtain a (possibly trivial) partition $\beta$ with $\sH(\beta) < \delta$. Say $A_1, A_2, \ldots, A_M \in \alpha$ were combined into a single class to form $\beta$. Set $B = A_1 \cup A_2 \cup \cdots \cup A_M$ so that $\{B\} = \beta \setminus \alpha$. Let $\alpha_M = \{A_1', A_2', \ldots, A_M'\}$ be any measurable partition of $X$ with $A_i \subseteq A_i'$ for each $i$. Notice that for any probability measure $\nu$ on $X$ (not necessarily $G$-invariant) and any $h \in G$ we have
$$\sH_\nu(h \cdot \alpha / h \cdot \beta) = \nu(h \cdot B) \cdot \sH_{\nu_{h \cdot B}}(h \cdot \alpha) = \nu(h \cdot B) \cdot \sH_{\nu_{h \cdot B}}(h \cdot \alpha_M)$$
$$\leq \nu(h \cdot B) \cdot \log(M) \leq \log(M).$$
Since the action is ergodic and not essentially free, the stabilizer of $\mu$-almost every $x \in X$ is non-trivial. By the previous lemma, there is $1_G \neq g \in G$ such that
$$\mu \bigg( \{x \in X \: g \cdot x \in \B{|g|-1}{S} \cdot x\} \bigg) > 1 - \delta / \log(M).$$

For $x \in X$, define $\psi(x)$ to be the $\preceq$-least element of $G$ satisfying $g \cdot x = \psi(x) \cdot x$. Note that $\mu(\psi^{-1}(g)) < \delta / \log(M)$. Since $\psi$ is measurable and its image is finite, it induces a finite measurable partition, $\xi$, of $X$. Let $0 < \kappa < \delta / \log(M)$ be such that
$$-\kappa \cdot \log(\kappa) - (1 - \kappa) \cdot \log(1 - \kappa) < \delta.$$
By Lemma \ref{LEM CONSHAN}, there exists $\rho > 0$ such that if $\chi$ is a countable measurable partition of $X$ satisfying $\sH(\xi / \chi) < \rho$, then there is a subcollection $\chi' \subseteq \chi$ consisting of positive measure sets such that $\mu(\cup \chi') > 1 - \kappa$ and for every $C \in \chi'$ there is $E \in \xi$ with $\mu_C(E) > 1 - \kappa$. Let $N > n + |g|$ be such that
$$\sH(\xi / \B{N-1}{S} \cdot \alpha) < \rho.$$
Such an $N$ exists since $\alpha$ is a generating partition.

Fix $w \in \Sp{n}{S}$. For $f \in G$ let $\delta(f, \alpha)$ be the independence decay at $f$ as defined in the paragraph preceding Theorem \ref{INTRO DISS}. Let $F(w)$ be the set of $w \neq f \in G$ for which $|f| = |f w^{-1}| + |w|$. In other words, $F(w)$ consists of the $w \neq f \in G$ whose reduced $S$-word representations end with the reduced $S$-word representation of $w$. We claim that
$$\sum_{f \in F(w)} \delta(f, \alpha) \geq \sH(w \cdot \alpha / \Pre(w) \cdot \alpha) - \frac{\epsilon}{|\Sp{n}{S}|}.$$
Fix any choice of $t \in G$ such that $g^{-1} t w \in \Sp{N}{S}$ and $|g^{-1} t w| = |g^{-1}| + |t| + |w|$. Then $g^{-1} t w \in F(w)$. Let $p_0 = w, p_1, p_2, \ldots, p_k = g^{-1} t w$ be the sequence of vertices in the path from $w$ to $g^{-1} t w$ in the left $S$-Cayley graph of $G$. Then $p_i \in F(w)$ for $i > 0$. By Lemma \ref{LEM PATH}
$$\sum_{f \in F(w)} \delta(f, \alpha) \geq \sum_{i = 1}^k \delta(p_i, \alpha) = \sH(w \cdot \alpha / \Pre(w) \cdot \alpha) - \sH(g^{-1} t w \cdot \alpha / \Pre(g^{-1} t w) \cdot \alpha).$$
So it suffices to show that $\sH(g^{-1} t w \cdot \alpha / \Pre(g^{-1} t w) \cdot \alpha) < \epsilon / |\Sp{n}{S}|$.

We have $\Pre(g^{-1} t w) \supseteq \B{N-1}{S}$ so
$$\sH(\xi / \Pre(g^{-1} t w) \cdot \alpha) \leq \sH(\xi / \B{N-1}{S} \cdot \alpha) < \rho.$$
Thus there is a collection $\C \subseteq \Pre(g^{-1} t w) \cdot \alpha$ consisting of positive measure sets such that $\mu(\cup \C) > 1 - \kappa$ and for every $C \in \C$ there is $E \in \xi$ with $\mu_C(E) > 1 - \kappa$. Let $\C' \subseteq \C$ consist of those $C \in \C$ for which there is $E \in \xi$ with $E \neq \psi^{-1}(g)$ and $\mu_C(E) > 1 - \kappa$. Observe that
$$\mu(\cup \C') > 1 - \kappa - \frac{\delta / \log(M)}{(1 - \kappa)} > 1 - \kappa - 2 \delta / \log(M)$$
since $\mu(\psi^{-1}(g)) < \delta / \log(M)$. Fix $C \in \C'$ and let $h \in \B{|g|-1}{S}$ be such that $\mu_C(\psi^{-1}(h)) > 1 - \kappa$. As $C \in \Pre(g^{-1} t w) \cdot \alpha$ and $h^{-1} t w \in \Pre(g^{-1} t w)$, we may fix an $A \in \alpha$ with $C \subseteq h^{-1} t w \cdot A$. Since $g^{-1} h$ acts trivially on $\psi^{-1}(h) \cap C$, we have
$$g^{-1} t w \cdot A = g^{-1} h \cdot (h^{-1} t w \cdot A) \supseteq g^{-1} h \cdot (\psi^{-1}(h) \cap C) = \psi^{-1}(h) \cap C.$$
Therefore
$$\mu_C(g^{-1} t w \cdot A) \geq \mu_C(\psi^{-1}(h)) > 1 - \kappa.$$
So $g^{-1} t w \cdot A$ covers most of $C$. If $A \in \beta$ then $\mu_C(g^{-1} t w \cdot B) < \kappa$ and therefore
$$\sH_{\mu_C}(g^{-1} t w \cdot \alpha / g^{-1} t w \cdot \beta) \leq \mu_C(g^{-1} t w \cdot B) \cdot \log(M) \leq \kappa \cdot \log(M).$$
On the other hand, if $A \not\in \beta$ then $A = A_i \subseteq A_i'$ for some $i$, $A \subseteq B$, and
$$\mu_{C \cap g^{-1} t w \cdot B}(g^{-1} t w \cdot A_i') = \mu_{C \cap g^{-1} t w \cdot B}(g^{-1} t w \cdot A) \geq \mu_C(g^{-1} t w \cdot A) > 1 - \kappa.$$
Thus in this case
$$\sH_{\mu_C}(g^{-1} t w \cdot \alpha / g^{-1} t w \cdot \beta) = \mu_C(g^{-1} t w \cdot B) \cdot \sH_{\mu_{C \cap g^{-1} t w \cdot B}}(g^{-1} t w \cdot \alpha_M)$$
$$\leq \sH_{\mu_{C \cap g^{-1} t w \cdot B}}(g^{-1} t w \cdot \alpha_M) < -\kappa \cdot \log(\kappa) - (1 - \kappa) \cdot \log(1 - \kappa) + \kappa \cdot \log(M - 1).$$
So in either case we have
$$\sH_{\mu_C}(g^{-1} t w \cdot \alpha / g^{-1} t w \cdot \beta) < -\kappa \cdot \log(\kappa) - (1 - \kappa) \cdot \log(1 - \kappa) + \kappa \cdot \log(M) < 2 \delta.$$
This holds for every $C \in \C'$. It follows that
$$\sH(g^{-1} t w \cdot \alpha / \Pre(g^{-1} t w) \cdot \alpha)$$
$$= \sH(g^{-1} t w \cdot \beta / \Pre(g^{-1} t w) \cdot \alpha) + \sH(g^{-1} t w \cdot \alpha / \Pre(g^{-1} t w) \cdot \alpha \vee g^{-1} t w \cdot \beta)$$
$$\leq \sH(\beta) + \sH(g^{-1} t w \cdot \alpha / \Pre(g^{-1} t w) \cdot \alpha \vee g^{-1} t w \cdot \beta)$$
$$< \delta + \sum_{C \in \Pre(g^{-1} t w) \cdot \alpha} \mu(C) \cdot \sH_{\mu_C}(g^{-1} t w \cdot \alpha / g^{-1} t w \cdot \beta)$$
$$< \delta + \mu(X \setminus \cup \C') \cdot \log(M) + \sum_{C \in \C'} \mu(C) \cdot \sH_{\mu_C}(g^{-1} t w \cdot \alpha / g^{-1} t w \cdot \beta)$$
$$< \delta + \left( \kappa + \frac{2 \delta}{\log(M)} \right) \log(M) + 2 \delta < 6 \delta = \frac{\epsilon}{|\Sp{n}{S}|}.$$
This justifies our claim.

Now we compute
$$\sum_{1_G \neq h \in G} \delta(h, \alpha) \geq \sum_{w \in \Sp{n}{S}} \sum_{f \in F(w)} \delta(f, \alpha)$$
$$\geq \sum_{w \in \Sp{n}{S}} \left( \sH(w \cdot \alpha / \Pre(w) \cdot \alpha) - \frac{\epsilon}{|\Sp{n}{S}|} \right) = \sH(\B{n}{S} \cdot \alpha / \B{n-1}{S} \cdot \alpha) - \epsilon > \frac{1}{\epsilon} - \epsilon.$$
By letting $\epsilon$ tend to zero we find that the cumulative independence decay is infinite. Therefore $f_G(X, \mu) = -\infty$ by Theorem \ref{THM DISS}.
\end{proof}

\begin{thm} \label{MAIN THEOREM}
Let $G$ have rank $r > 1$ and let $G$ act on a probability space $(X, \mu)$. Assume that $f_G(X, \mu)$ is defined. If $f_G(X, \mu) \neq - \infty$ and $(Y, \nu)$ is any factor of $(X, \mu)$, then for $\nu$-almost every $y \in Y$, the stabilizer of $y$ is either trivial or has finite index in $G$. Furthermore, $\nu$-almost every $y \in Y$ with non-trivial stabilizer is an atom, and thus there are essentially only countably many points with non-trivial stabilizer.
\end{thm}

\begin{proof}
To be clear, when we say that $\nu$-almost every $y \in Y$ with non-trivial stabilizer is an atom, we mean that there is a Borel set $Y' \subseteq Y$ such that $\nu(Y') = 1$ and every $y \in Y'$ with non-trivial stabilizer is an atom.

First suppose that $G \acts (Y, \nu)$ is ergodic. If $\nu$ is purely atomic then by ergodicity there are only finitely many atoms and they lie in a single orbit. In particular their stabilizers have finite index in $G$. So suppose that $\nu$ is not purely atomic. By ergodicity $\nu$ has no atoms. Let $\alpha$ be a generating partition for $G \acts (X, \mu)$ with $\sH_\mu(\alpha) < \infty$. Let $n \in \N$ be such that
$$-(r - 1) \log(n) < f_G(X, \mu) - \sH_\mu(\alpha).$$
Since $\nu$ has no atoms, we can find a measurable partition $\beta$ of $Y$ which has precisely $n$ classes of positive measure. The sub-$\sigma$-algebra $G \cdot \beta$ gives rise to a factor map $(Y, \nu) \mapsto (Y', \nu')$. Moreover, $\beta$ pushes forward to a partition $\beta'$ of $(Y', \nu')$, and $\beta'$ has $n$ classes of positive measure. Also $\beta'$ is a generating partition for $G \acts (Y', \nu')$ and thus $f_G(Y', \nu')$ is defined since $\beta'$ is finite. By Theorem \ref{BOWEN FACTOR},
$$f_G(Y', \nu') \geq f_G(X, \mu) - \sH_\mu(\alpha) > - (r - 1) \log(n).$$
By Proposition \ref{MAIN PROP}, either $G \acts (Y', \nu')$ is essentially free or $\nu'$ is purely atomic. If $\nu'$ were purely atomic, then it would have to have at least $n$ atoms since each class of $\beta'$ has positive measure. Every atom would have the same measure by ergodicity, so by Lemma \ref{LEM FINACT} we would have $f_G(Y', \nu') \leq -(r - 1) \log(n)$, a contradiction. So $G \acts (Y', \nu')$ is essentially free. It follows that $G \acts (Y, \nu)$ must be essentially free as well.

Now consider the general case in which $G \acts (Y, \nu)$ is not ergodic. Let $\tau$ be the ergodic decomposition of $\mu$. By Corollary \ref{COR ERGDEC}, $\tau$ is purely atomic. Let $E \subseteq \E(X)$ be the set of atoms of $\tau$. By Theorem \ref{THM ERGDEC}, $f_G(X, \lambda)$ is defined for every $\lambda \in E$ and
$$f_G(X, \mu) = \sum_{\lambda \in E} \tau(\{\lambda\}) \cdot f_G(X, \lambda) - (r - 1) \sH(\tau).$$
Since $f_G(X, \mu) \neq -\infty$, we must have $f_G(X, \lambda) \neq -\infty$ for every $\lambda \in E$. Each measure $\lambda \in E$ pushes forward to an ergodic measure $\lambda^*$ on $Y$, and we have
$$\nu = \sum_{\lambda \in E} \tau(\{\lambda\}) \cdot \lambda^*.$$
By the previous paragraph $\lambda^*$ is either purely atomic or else $G \acts (Y, \lambda^*)$ is essentially free. Of course, when $\lambda^*$ is purely atomic the stabilizer of $\lambda^*$-almost every point has finite index in $G$. So if we let $B$ be the set of all $y \in Y$ for which the stabilizer of $y$ is neither trivial nor has finite index in $G$, then $\lambda^*(B) = 0$ for every $\lambda \in E$. By the above decomposition of $\nu$, $\nu(B) = 0$ and thus the stabilizer of $\nu$-almost every $y \in Y$ is either trivial or has finite index in $G$. Now let $A \subseteq Y$ be the set of all atoms of $\nu$. So $y \in A$ if and only if $\nu(\{y\}) > 0$. Also let $C$ be the set of $y \in Y$ for which the stabilizer of $y$ is non-trivial. Then $\lambda^*(C \setminus A) = 0$ for every $\lambda \in E$ and thus $\nu(C \setminus A) = 0$. So setting $Y' = Y \setminus (C \setminus A)$, we have $\nu(Y') = 1$ and every $y \in Y'$ with non-trivial stabilizer lies in $Y' \cap C \subseteq A$ and is thus an atom of $\nu$.
\end{proof}

We now prove Corollary \ref{INTRO SUBERG} from the introduction.

\begin{cor}
Let $G$ have rank $r > 1$ and let $G$ act on a probability space $(X, \mu)$. Assume that $f_G(X, \mu)$ is defined. If $G \acts (X, \mu)$ is ergodic and $f_G(X, \mu) \neq -\infty$, then there is $n \in \N$ such that for every subgroup $\Gamma \leq G$ containing a non-trivial normal subgroup of $G$ the number of ergodic components of $\Gamma \acts (X, \mu)$ is at most $n$.
\end{cor}

\begin{proof}
Let $\alpha$ be a generating partition for $G \acts (X, \mu)$ with $\sH(\alpha) < \infty$. Let $n$ be maximal such that
$$-(r - 1) \log(n) \geq f_G(X, \mu) - \sH(\alpha).$$
Such an $n$ exists since we are assuming $r > 1$. Let $\Gamma \leq G$ be a subgroup containing a non-trivial normal subgroup of $G$. Say $K \neq \{1_G\}$ is normal in $G$ and $K \leq \Gamma$. Let $\tau_K$ be the ergodic decomposition of $K \acts (X, \mu)$. Then $G$ acts on $(\M_K(X), \tau_K)$ and this action is a factor of $G \acts (X, \mu)$ (the factor map is induced by the sub-$\sigma$-algebra of all $K$-invariant Borel sets). For $\tau_K$-almost every $\lambda \in \M_K(X)$ the $G$-stabilizer of $\lambda$ contains $K$ and is thus non-trivial. So Theorem \ref{MAIN THEOREM} implies that $\tau_K$ is purely atomic. By ergodicity there are finitely many atoms, each with the same measure. So if there are $m$ atoms of $\tau_K$ then by Theorem \ref{BOWEN FACTOR} and Lemma \ref{LEM FINACT}
$$- (r - 1) \log(m) = f_G(\M_K(X), \tau_K) \geq f_G(X, \mu) - \sH(\alpha).$$
Therefore $m$, the number of ergodic components of $K \acts (X, \mu)$, is at most $n$. Since $K \leq \Gamma$, every $\Gamma$ ergodic component of $\Gamma \acts (X, \mu)$ contains at least one $K$-ergodic component. Therefore the number of $\Gamma$-ergodic components is at most $n$.
\end{proof}

The above corollary cannot be extended to hold for all non-trivial subgroups of $G$. This is demonstrated by the following example provided by Lewis Bowen (private communication).

\begin{prop}[Lewis Bowen] \label{BOWEN EX}
Let $G$ have rank $r > 1$. Then there exists an ergodic action $G \acts (X, \mu)$ and a non-trivial subgroup $\Gamma \leq G$ such that $f_G(X, \mu)$ is defined and finite (i.e. $f_G(X, \mu) \neq -\infty$), but there are infinitely many ergodic components of $\Gamma \acts (X, \mu)$.
\end{prop}

\begin{proof}
Let $\preceq$ be a well ordering of $G$ as described in the paragraph above Theorem \ref{INTRO DISS}, and let $\Pre(g)$ be the set of group elements strictly preceding $g$. Fix $t \in S$ and set $\Gamma = \langle t \rangle$. Let $G / \Gamma = \{ g \Gamma \: g \in G\}$ be the set of all left cosets of $\Gamma$.

Let $\nu$ be the probability measure on the positive integers, $\N_+$, defined by $\nu(\{n\}) = 2^{-n}$. It is easily computed that $\sH(\nu) = 2 \cdot \log(2) < \infty$. Let $X$ be the set of all functions from $G / \Gamma$ to $\N_+$. Equivalently,
$$X = \N_+^{G / \Gamma} = \prod_{g \Gamma \in G / \Gamma} \N_+.$$
We let $\mu = \nu^{G / \Gamma}$ be the product measure. We let $G$ act on $X$ by permuting coordinates on the left: $(h \cdot x)(g \Gamma) = x(h^{-1} g \Gamma)$. Since $G$ acts by permuting coordinates, it readily follows that $\mu$ is $G$-invariant. The action $G \acts (X, \mu)$ is quite similar to a Bernoulli shift. After a minor and obvious modification, the standard argument that Bernoulli shifts are ergodic shows that $G \acts (X, \mu)$ is ergodic as well.

Let $\alpha = \{A_i \: i \in \N_+\}$ be the partition of $X$ given by $A_i = \{x \in X \: x(\Gamma) = i\}$. Then $\alpha$ is a generating partition for $G \acts (X, \mu)$ and $\sH_\mu(\alpha) = \sH(\nu) < \infty$. So $f_G(X, \mu)$ is defined. Notice that $\gamma \cdot \alpha = \alpha$ for every $\gamma \in \Gamma$. Since values at distinct coordinates are independent, it is not difficult to check that $\sH(g \cdot \alpha / \Pre(g) \cdot \alpha)$ is $0$ if $g \Gamma \subseteq \Pre(g) \cdot \Gamma$ and is otherwise equal to $\sH(\alpha)$. Since $t \in S$ and $\Gamma = \langle t \rangle$, it follows that $\sH(g \cdot \alpha / \Pre(g) \cdot \alpha)$ equals $0$ if the reduced $S$-word representation of $g$ ends with $t$ or $t^{-1}$ and is otherwise equal to $\sH(\alpha)$. So if $1_G \neq g \in G$ and $s \in S$ satisfy $|s^{-1} g| = |g| - 1$, then $\sH(g \cdot \alpha / \Pre(g) \cdot \alpha) = \sH(s^{-1} g \cdot \alpha / \Pre(s^{-1} g) \cdot \alpha)$ unless $g \in \{t, t^{-1}\}$. It follows that $\delta(g, \alpha) = 0$ for $g \not\in \{1_G, t, t^{-1}\}$ and $\delta(t, \alpha) = \delta(t^{-1}, \alpha) = \sH(\alpha)$. So $f_G(X, \mu) = 0$ by Theorem \ref{THM DISS} (alternatively, one could observe that $G \acts (X, \mu)$ is measurably conjugate to a Markov process and apply \cite[Theorem 11.1]{B10d}). So $G \acts (X, \mu)$ is ergodic, $f_G(X, \mu)$ is defined, and $f_G(X, \mu) = 0 \neq -\infty$. Finally, we have $\gamma \cdot A = A$ for every $A \in \alpha$ and $\gamma \in \Gamma$, so there are at least as many ergodic components of $\Gamma \acts (X, \mu)$ as there are members of $\alpha$ with positive measure. We conclude that there are infinitely many ergodic components of $\Gamma \acts (X, \mu)$.
\end{proof}

\end{document}